\documentclass[11pt, oneside, reqno]{amsart}
\usepackage{amsfonts, amssymb, color}
\usepackage{amsmath}
\usepackage{graphicx}
\usepackage{amscd}
\usepackage{amsthm}
\usepackage{tikz}
\usetikzlibrary{decorations.markings}
\usetikzlibrary{shapes,arrows,positioning}	
\usetikzlibrary {arrows.meta}	
\usetikzlibrary{decorations.markings,decorations.pathmorphing}

\newcommand{\rmidarrow}{\tikz \draw[thick,-{Computer Modern Rightarrow}] (-1pt,0) --(1pt,0);}	
\newcommand{\lmidarrow}{\tikz \draw[thick, -{Computer Modern Rightarrow}] (1pt,0) -- (-1pt,0);}

\usepackage[margin=1in]{geometry}

\begin{document}

\newtheorem{thm}{Theorem}[section]
\newtheorem{lem}[thm]{Lemma}
\newtheorem{prop}[thm]{Proposition}
\newtheorem{cor}[thm]{Corollary}
\newtheorem{obs}[thm]{Observation}
\newtheorem{rem}[thm]{Remark}
\newtheorem{Def}[thm]{Definition}
\newtheorem{ex}[thm]{Example}

\title{Quantum Automorphism Group of Direct Sum of Cuntz Algebras}
\author{Ujjal Karmakar \MakeLowercase{and} Arnab Mandal }
\maketitle

\begin{abstract}

In this article, we explore the quantum symmetry of the direct sum of a finite family of Cuntz algebras $\{\mathcal{O}_{n_i} \}_{i=1}^{m}$, viewing them as graph $C^*$-algebras associated to the graphs $\{L_{n_i}\}_{i=1}^{m}$ (where $L_n$ denotes the graph containing $n$ loops based at a single vertex), in the category introduced by Joardar and Mandal in \cite{Mandal}. It has been shown that the quantum automorphism group of the direct sum of non-isomorphic Cuntz algebras is ${U}_{n_1}^{+}*{U}_{n_2}^{+}* \cdots *{U}_{n_m}^{+}$ for distinct $n_i$'s, i.e.
	\begin{equation*}
	Q_{\tau}^{Lin}(\sqcup_{i=1}^{m} ~ L_{n_i}) \cong  *_{i=1}^{m} ~~ Q_{\tau}^{Lin}(L_{n_i}) \cong {U}_{n_1}^{+}*{U}_{n_2}^{+}* \cdots *{U}_{n_m}^{+},	\end{equation*}
	where $Q_{\tau}^{Lin}(\Gamma)$ denotes the quantum automorphism group of the graph $C^*$-algebra associated to $\Gamma$.
	Also, the quantum automorphism group of the direct sum of $m$ copies of isomorphic Cuntz algebra $\mathcal{O}_n$ is  $U_n^+ \wr_* S_m^+,$ i.e.
	\begin{equation*}
	Q_{\tau}^{Lin}(\sqcup_{i=1}^{m} ~ L_n) \cong  Q_{\tau}^{Lin}(L_n) \wr_* S_m^+ \cong  U_n^+ \wr_* S_m^+.
	\end{equation*}
	Furthermore, we have provided counter-examples to demonstrate that the isomorphisms mentioned above cannot be generalized to arbitrary graph $C^*$-algebras, whereas analogous relations can be extended in the context of quantum automorphism groups of graphs in the sense of Banica and Bichon.	
\end{abstract}


\section{Introduction}
The Cuntz algebra, introduced by Joachim Cuntz in 1977 \cite{Cuntz}, is an interesting example of a $C^*$-algebra that can often be described as a universal $C^*$-algebra in terms of generators and algebraic relations. Formally, the Cuntz algebra with $n \in \mathbb{N}$ generators (denoted by $\mathcal{O}_{n}$) is the universal $C^*$-algebra with generators $s_1,s_2,...,s_n$ having relations $ s_{i}^*s_{i}=1$  for all $i \in \{1,2,...,n\}$ (i.e., all $s_{i}$'s are isometries) and $ \sum_{i=1}^{n} s_{i}s_{i}^*=1 $. By construction, $\mathcal{O}_{n}$ is commutative if and only if  $n=1$ and $\mathcal{O}_{1}$ is isomorphic to $C(S^1)$. Moreover, $\mathcal{O}_{n}$ and $\mathcal{O}_{m}$ are isomorphic iff $m=n$. Interestingly, the Cuntz algebras can be viewed as graph $C^*$-algebras. A graph $C^*$-algebra (denoted by $C^*(\Gamma)$) is constructed as a universal $C^*$-algebra generated by some orthogonal projections and partial isometries coming from a given directed graph $\Gamma$. Though not every $C^*$-algebra can be thought of as a graph $C^*$-algebra, for instance $C(S^1 \times S^1)$, many important examples of $C^*$-algebras, including matrix algebras, Cuntz algebras, Toeplitz algebra etc. can be recognized as graph $C^*$-algebras. In particular, $\mathcal{O}_{n}$ can be realized as a graph $C^*$-algebra with respect to the graph containing $n$ distinct loops based at a single vertex. Moreover, a graph $C^*$-algebra is a Cuntz-Krieger algebra (introduced in \cite{CK}) if the underlying finite graph contains no sink (\cite{KPRR}). An importance of a graph $C^*$-algebra is that some graph-theoretic properties of the underlying graph $\Gamma$ can be recovered from the operator algebraic properties of $C^*(\Gamma)$ and vice-versa (see \cite{Kumjian}, \cite{Tom}, \cite{Raeburn} for more details).\\

On the other hand, compact quantum groups (in short, CQG) appeared in mathematics around the 20th century, almost a hundred years after the appearance of the group. Using an analytic construction, S.L. Woronowicz constructed a few initial examples of CQGs in \cite{Wor}. In non-commutative geometry, mathematicians were always highly determined to capture an appropriate notion of symmetry for the non-commutative spaces. The aim was to generalize the concept of classical group symmetry to develop a `non-commutative version of symmetry' \cite{connes}. In 1998, Shuzhou Wang formulated the notion of a quantum automorphism group for a finite space $X_{n}$ (a space containing $n$ points) as the `universal object' in a category of compact quantum groups acting on the unital $C^*$-algebra $C(X_n)$. He showed that the quantum automorphism group of $X_n$ is larger than the classical automorphism group of $X_n$ for $n>3$ and its associated $C^*$-algebra is infinite-dimensional and non-commutative. Moreover, the quantum automorphism groups for any finite-dimensional $C^*$-algebras were classified by him (see \cite{Wang}). In 2003, the quantum symmetry structure for a finite graph $\Gamma$ was formulated by J. Bichon in \cite{Bichon} and we denote that quantum automorphism group of $\Gamma$ by $QAut_{Bic}(\Gamma)$. Later, the quantum symmetry for a finite graph $\Gamma$  was also extended by T. Banica in \cite{Ban} and the quantum automorphism group of $\Gamma$ (in the sense of Banica) is denoted by $QAut_{Ban}(\Gamma)$. The notion of quantum isometry group (in an infinite-dimensional set-up) was formulated by Goswami in \cite{Laplace}. A few years later, adopting the key ideas from their works, T. Banica and A. Skalski proposed the notion of orthogonal filtration on a $C^*$-algebra, equipped with a faithful state, and its quantum symmetry \cite{Ortho}.
Since the associated graph $C^*$-algebra over a finite, directed graph may indeed be infinite-dimensional, in 2017, an interesting study about the quantum symmetry of a graph $C^*$-algebra $C^*(\Gamma)$ was investigated by S. Schmidt and M. Weber in \cite{Web}. Inspired by their work, in 2018-2021, S. Joardar and A. Mandal studied the quantum symmetry of a graph $C^*$-algebra $C^*(\Gamma)$ for a finite graph $\Gamma$ without isolated vertices, in a categorical framework considering two different categories, namely $\mathfrak{C}_{\tau}^{Lin}(\Gamma)$ and $\mathfrak{C}_{KMS}^{Lin}(\Gamma)$ introduced in \cite{Mandal} and \cite{Mandalkms} respectively. Though the quantum symmetry group $QAut_{SW}(\Gamma)$ of the graph $C^*$-algebra coincides with the quantum automorphism group (of the underlying graph) $QAut_{Ban}(\Gamma)$; ${Q}_{\tau}^{Lin}(\Gamma)$ is strictly larger than $QAut_{Ban}(\Gamma)$. Moreover, for a finite graph $\Gamma$ without a sink, one can also consider the category $\mathfrak{C}_{KMS}^{Lin}(\Gamma)$ and under certain conditions, the category $\mathfrak{C}_{\tau}^{Lin}(\Gamma)$ coincides with $\mathfrak{C}_{KMS}^{Lin}(\Gamma)$ (see \cite{Mandalkms}).\\

\noindent Now, if we take $n$ graphs $ \{\Gamma_{i}\}_{i=1}^{n}$ and consider the disjoint union of those graphs, i.e. $\sqcup_{i=1}^{n} ~ \Gamma_{i}$, then it is natural to ask what the relation is between the quantum symmetry of the graph ($ \sqcup_{i=1}^{n} ~ \Gamma_{i}$) and the individual quantum symmetries of the graphs $\Gamma_{i}$'s in the sense of Banica or Bichon, i.e., is there any relation between $QAut_{Ban}(\sqcup_{i=1}^{n} ~ \Gamma_{i})$ [$QAut_{Bic}(\sqcup_{i=1}^{n} ~ \Gamma_{i})$] and $ \{QAut_{Ban}(\Gamma_{i})\}_{i=1}^{n} $ [$ \{QAut_{Bic}(\Gamma_{i})\}_{i=1}^{n}$]? It is well known that if $\Gamma_{i}$'s are connected and mutually `quantum non-isomorphic', then
\begin{equation*}
QAut_{Ban}(\sqcup_{i=1}^{n} ~ \Gamma_{i}) \cong *_{i=1}^{n} ~~ QAut_{Ban}(\Gamma_{i})
\end{equation*} (refer to \cite{Lupini}, \cite{SSthesis}, \cite{QAut_tree} and \cite{Meu} for details). Moreover, if all $\Gamma_i$'s are connected and isomorphic to each other, then 
\begin{equation*}
QAut_{Ban}(\sqcup_{i=1}^{n} ~ \Gamma_{i}) \cong QAut_{Ban}(\Gamma_1) \wr_* S_n^+
\end{equation*} 
(\cite{Banver}) and 
\begin{equation*}
QAut_{Bic}(\sqcup_{i=1}^{n} ~ \Gamma_{i}) \cong  QAut_{Bic}(\Gamma_1) \wr_* S_n^+
\end{equation*}
(\cite{Bichonwreath}). Similar questions also arise in a graph $C^*$-algebraic context, i.e. what is the relation among $Q_{\tau}^{Lin}(\sqcup_{i=1}^{n} ~ \Gamma_{i})$ and $\{ Q_{\tau}^{Lin}(\Gamma_{i})\}_{i=1}^{n}$? Do analogous results hold for the quantum symmetries of graph $C^*$-algebras? More precisely, is
(i) $Q_{\tau}^{Lin}(\sqcup_{i=1}^{n} ~ \Gamma_{i}) \cong Q_{\tau}^{Lin}(\Gamma_{1})* \cdots * Q_{\tau}^{Lin}(\Gamma_{n})$ if $\Gamma_{i}$'s are mutually non-isomorphic graphs?
(ii) $ Q_{\tau}^{Lin}(\sqcup_{i=1}^{n} ~ \Gamma_{i}) \cong  Q_{\tau}^{Lin}(\Gamma_1) \wr_* S_n^+$ if all  $\Gamma_i$'s are isomorphic?
However, the answers are negative in both of the scenarios. We have encountered non-isomorphic (even `quantum non-isomorphic') graphs, specifically $P_1$ and $So_2$ (see Figure \ref{P1So2} in Subsection \ref{counternonisom} for details), for which relation (i) does not hold in general. Interestingly, relation (ii) also does not hold when considering two disjoint copies of $P_1$ (see Figure \ref{P1P1}). But we have also found out a class of graph $C^*$-algebras, namely Cuntz algebras $\mathcal{O}_n$ (whose underlying graph is $L_n$), for which both relations (i) and (ii) hold in the sense that
\begin{equation*}
Q_{\tau}^{Lin}(\sqcup_{i=1}^{m} ~ L_{n_i}) \cong  *_{i=1}^{m} ~~ Q_{\tau}^{Lin}(L_{n_i}) \cong {U}_{n_1}^{+}*{U}_{n_2}^{+}* \cdots * {U}_{n_m}^{+},
\end{equation*}
if $L_{n_i}$'s are mutually non-isomorphic graphs (i.e., $n_i$'s are distinct) and 
\begin{equation*}
Q_{\tau}^{Lin}(\sqcup_{i=1}^{m} ~ L_n) \cong  Q_{\tau}^{Lin}(L_n) \wr_* S_m^+ \cong  U_n^+ \wr_* S_m^+.
\end{equation*}


Now, we briefly discuss the presentation of this article as follows: In the 2nd section, some pre-requisites are recalled about a directed graph, graph $C^*$-algebras, compact quantum groups and their action on a $C^*$-algebra, orthogonal filtration, quantum automorphism groups etc. Moreover, we recall the quantum symmetry of a graph $C^*$-algebra in the category introduced in \cite{Mandal} and \cite{Mandalkms}. We consider three categories on the direct sum of Cuntz algebras: (i) the category $\mathfrak{C}_{\tau}^{Lin}$  from \cite{Mandal}, (ii) a modified category of KMS state on the direct sum of Cuntz algebras, and (iii) the orthogonal filtration preserving category. In the 3rd section, we describe the quantum automorphism group of the direct sum of non-isomorphic Cuntz algebras in the same categories mentioned above. Additionally, we provide a counter-example to demonstrate that an analogous relation cannot be extended in general. In the 4th section, we explore the quantum automorphism group of isomorphic Cuntz algebras and provide a counter-example to illustrate that the same formula cannot be extended for an arbitrary graph $C^*$-algebra.

\section{Preliminaries}

\subsection{Notations and conventions}
For a set $X$,  $|X|$ will denote the cardinality of $X$ and $ id_{X} $ will denote the identity function on $X$. The $n$-set $\{1,2,3,...,n\}$ will be denoted by $[n]$. $ I_{n \times n}$ is the identity matrix on $ M_n(\mathbb{C})$. For a $C^*$-algebra $ \mathcal{B} $, $\mathcal{B}^*$  is the set of all linear bounded functionals on $ \mathcal{B} $. For a set $X$, span($X$) will denote the linear space spanned by the elements of $X$. The tensor product `$\otimes$' denotes the spatial or minimal tensor product between two $C^*$-algebras.\\
For us, all the $C^*$-algebras are unital.

\subsection{The Cuntz algebra }
The Cuntz algebra $\mathcal{O}_{n}$ was introduced by Cuntz in 1977 (\cite{Cuntz}).
An element $s$ in a $C^*$-algebra is an isometry if $s^*s = 1 .$
\begin{Def}
 The Cuntz algebra (with $n \in \mathbb{N}$ generators) $\mathcal{O}_{n}$ is the universal $C^*$-algebra generated by isometries $s_1,s_2,...,s_n$ such that $\sum_{i=1}^{n} s_{i}s_{i}^*=1 $,  i.e. $\mathcal{O}_{n}:= C^*\{s_1,s_2,...,s_n |  ~~s_{i}^*s_{i}=1 ~~\forall i \in [n], \sum_{i=1}^{n} s_{i}s_{i}^*=1\}$.\\
 Clearly, $\mathcal{O}_{1}$ is $C^*$-isomorphic to $C(S^1)$.
\end{Def}
\noindent  Moreover, since $K_0(\mathcal{O}_{n})= \mathbb{Z}_{n-1} $ for all $n \in \mathbb{N}$, $\mathcal{O}_{n}$ and $\mathcal{O}_{m}$ are isomorphic iff $m=n$ \cite{CuntzK}.\\

\noindent An important fact is that one can also view the Cuntz algebra (with $n$ generators) $\mathcal{O}_{n}$ as a graph $C^*$-algebra. Next, we will introduce graph $C^*$-algebra for some finite, directed graph.\\
A directed graph $\Gamma=\{ V(\Gamma), E(\Gamma),s,r \}$ consists of countable sets $ V(\Gamma) $ of vertices and $ E(\Gamma)$ of edges together with the maps $s,r: E(\Gamma) \to V(\Gamma) $ describing the source and range of the edges. We say that a vertex $v \in V(\Gamma)$ is \textbf{adjacent to} $w \in V(\Gamma)$ (denoted by $v \to w $)  if there exists an edge $e \in E(\Gamma)$ such that  $ v=s(e)$ and $ w=r(e) $. A graph is said to be \textbf{finite} if both $|V(\Gamma)|$ and $|E(\Gamma)|$ are finite. A directed {\bf graph without isolated vertices} means for every vertex $v \in V(\Gamma)$, at least one of $s^{-1}(v)$ and $r^{-1}(v)$ is non-empty.  A \textbf{directed path} $\alpha$ of length $n$ in a directed graph $\Gamma$ is a sequence $\alpha=e_{1}e_{2} \cdots e_{n} $ of edges in $\Gamma$ such that $ r(e_{i}) = s(e_{i+1}) $ for $1 \leq i \leq (n-1) $. We define $s(\alpha):=s(e_{1})$ and $r(\alpha):=r(e_{n})$. $E^{< \infty}(\Gamma)$ denotes the set of all finite length paths on $\Gamma$.
A \textbf{loop} is a graph with a vertex $v$ and an edge $e$ such that $s(e)=r(e)=v$.
Let $\Gamma=\{V(\Gamma),E(\Gamma),s,r \}$ be a finite, directed graph with $|V(\Gamma)|=n$. The {\bf adjacency matrix } of $\Gamma$ with respect to the ordering of the vertices $ (v_{1},v_{2},..., v_{n}) $ is a matrix $A(\Gamma)= (a_{ij})_{i,j= 1,2,...,n} $ with 
$ a_{ij} =
\begin{cases}
n(v_{i},v_{j}) & if ~~ v_{i} \to v_{j} \\
0 &  ~~ otherwise 
\end{cases} $ where $ n(v_{i},v_{j})$ denotes the number of edges joining $v_{i}$ to $v_{j}$.\\ 

\noindent In this article, we will define graph $C^*$-algebra only for a finite, directed graph. For more details about the theory of graph $C^*$-algebras, consult \cite{BHRS}, \cite{BPRS}, \cite{KPRR}, \cite{Kumjian}, \cite{Pask}, \cite{Raeburn} and \cite{MS}.  
\begin{Def} \label{Def_Graph C* algebra}
	Given a finite, directed graph $\Gamma$, the graph $C^{*}$-algebra $C^{*}(\Gamma)$ is a universal $C^{*}$-algebra generated by  orthogonal projections $ \{p_{v}: v \in V(\Gamma) \}$ and partial isometries $ \{S_{e}: e \in E(\Gamma) \} $  such that \vspace{0.1cm}
	\begin{itemize}
		\item[(i)] $S_{e}^{*}S_{e}=p_{r(e)}$ for all $ e \in E(\Gamma) $. \vspace{0.1cm}
		\item[(ii)] $p_{v}=\sum\limits_{\{f:s(f)=v\}}S_{f}S_{f}^{*}$ for all $ v \in V(\Gamma) $ if $s^{-1}(v) \neq \emptyset $.\\
	\end{itemize}
\end{Def} 
Examples:
\begin{enumerate}
	\item The Cuntz algebra (with $n$ generators) $\mathcal{O}_{n}$ can be thought of as a graph $C^*$-algebra with respect to  $L_n$ (Figure \ref{Ln}), a graph containing $n$ loops based at a single vertex, i.e. $C^*(L_n)$ is $C^*$-isomorphic to $\mathcal{O}_n$. 
	
\begin{center}
\begin{figure}[htpb]
\begin{tikzpicture}
\draw[fill=black] (0,0) circle (2pt) node[anchor=south]{$v_{1}$};
\draw[black,thick](0,0.5) circle (0.5) node[above=0.25]{\rmidarrow}node[above=0.5]{$e_{11}$};
\draw[black,thick](0,0.75) circle (0.75) node[above=0.5]{\rmidarrow} node[above=0.75]{$e_{22}$};
\draw[black,thick](0,1.25) circle (1.25) node[above=1]{\rmidarrow}node[above=1.25]{$e_{nn}$};
\draw[loosely dotted, black,thick] (0,2)--(0,2.5);
\end{tikzpicture}
\caption{$L_{n}$}
\end{figure} \label{Ln}
\end{center}

	\item For a directed path $P_n$ (Figure \ref{Pn}) of length $n$ containing  $(n+1)$ vertices, graph $C^*$-algebra $C^*(P_n)$ is isomorphic to the $C^*$-algebra $M_{n+1}(\mathbb{C})$.
	\begin{center}
		\begin{figure}[htb]
			\begin{tikzpicture}
			\draw[fill=black] (0,0) circle (2pt) node[anchor=north]{$v_{1}$};
			\draw[fill=black] (2,0) circle (2pt) node[anchor=north]{$v_{2}$};
			\draw[fill=black] (4,0) circle (2pt) node[anchor=north]{$v_{3}$};
			\draw[fill=black] (6,0) circle (2pt) node[anchor=north]{$v_{n}$};
			\draw[fill=black] (8,0) circle (2pt) node[anchor=north]{$v_{n+1}$};
			
			\draw[black,thick] (0,0)-- node{\rmidarrow}  node[above]{$e_{12}$} (2,0);
			\draw[black,thick] (2,0)-- node{\rmidarrow} node[above]{$e_{23}$} (4,0);
			\draw[loosely dotted, black,thick] (4.5,0)-- (5.5,0);
			\draw[black,thick] (6,0)--node{\rmidarrow}  node[above]{$e_{n(n+1)}$} (8,0);
			\end{tikzpicture}
			\caption{$P_{n}$}
		\end{figure}  \label{Pn}
	\end{center}
	
\end{enumerate}

For any graph $C^{*}$-algebra $C^*(\Gamma)$, we have the following results.
\begin{prop} \label{graphprop}
Let $\Gamma=\{V(\Gamma),E(\Gamma),s,r \}$ be a finite, directed graph. For a path $\gamma=e_1e_2...e_n \in E^{< \infty}(\Gamma)$, we define $S_{\gamma}:=S_{e_{1}}S_{e_{2}}...S_{e_{n}} $, where $e_1,...,e_n \in E(\Gamma)$.
	\begin{itemize}
		\item[(i)] $S_{\gamma}^{*}S_{\mu}=0$  $\forall \gamma, \mu \in E^{< \infty}(\Gamma)$ with $\gamma \neq \mu $. In particular, $S_{e}^{*}S_{f}=0$  $\forall e,f \in E(\Gamma)$ with $e \neq f $. \vspace{0.1cm}
		\item[(ii)] $\sum\limits_{v \in V(\Gamma)}p_{v}=1$.\vspace{0.1cm}
		\item[(iii)] $ S_{e}S_{f}\neq 0 \Leftrightarrow r(e)=s(f)$, i.e. $ef \text{ is a path of length 2} $.     \\
		Moreover, $S_{e_{1}}S_{e_{2}}...S_{e_{k}} \neq 0 \Leftrightarrow r(e_{i})=s(e_{i+1}) $ for $i=1,2,...,(k-1)$, i.e. $S_{\gamma} \neq 0$ for all $\gamma \in E^{< \infty}(\Gamma)$.\vspace{0.1cm}
		\item[(iv)] $ S_{\gamma}S_{\mu}^{*}\neq 0 \Leftrightarrow r(\gamma)=r(\mu)$ $\forall \gamma, \mu \in E^{< \infty}(\Gamma)$ . In particular, for all $e,f \in E$, $ S_{e}S_{f}^{*}\neq 0 \Leftrightarrow r(e)=r(f).$ \vspace{0.1cm}
		\item[(v)] $p_{s(e)}S_e=S_e p_{r(e)}=S_e $ for $e \in E(\Gamma)$. \vspace{0.1cm}
		\item[(vi)] $ span\{S_{\gamma}S_{\mu}^{*} : \gamma, \mu \in E^{< \infty}(\Gamma) \text{ with } r(\gamma)=r(\mu)\} $ is dense in $C^{*}(\Gamma)$.\\
	\end{itemize} 
\end{prop}

\subsection{Direct sum of $C^*$-algebras}
Direct sum of $C^*$-algebras $A_1, A_2,...,A_n$ is a $C^*$-algebra whose underlying set is $A_1 \oplus A_2 \oplus \cdots  \oplus A_n := \{(a_1,a_2,...,a_n)|a_i \in A_i ~~ \forall i \in [n] \}$ together with coordinate-wise operations and the unique $C^*$-norm is given by $\Vert (a_1,a_2,...,a_n) \Vert = max \{\Vert a_i \Vert : i \in [n] \}$.\\
If all $A_i$'s are unital C*-algebras with unities $1_{A_i}$ respectively, then $(1_{A_1},1_{A_2},...,1_{A_n} )$ is the unity of $A_1 \oplus A_2 \oplus \cdots  \oplus A_n $.
\subsubsection{\underline{A state on direct sum of $C^*$-algebras}}
Since every $C^*$-algebra always admits a state, we can naturally define a state on the direct sum. Let $\phi_i$ be a state on $C^*$-algebra $A_i$ for all $ i\in [n]$. We can naturally define a state $\oplus_{i=1}^{n}\phi_{i} $ on $ A_1 \oplus A_2 \oplus \cdots  \oplus A_n  $ by $\oplus_{i=1}^{n}\phi_{i}~~(a_1,a_2,...,a_n)= \frac{1}{n}[\phi_1 (a_1) + \phi_2 (a_2)+ \cdots + \phi_n (a_n)]$.
\subsubsection{\underline{Direct sum of graph $C^*$-algebras}}
 Let $\{ \Gamma_{i} \}_{i=1}^{m} $ be disjoint graphs. Then 
 \begin{equation*}
 C^*(\sqcup_{i=1}^{m} ~~ \Gamma_{i}) \cong \oplus_{i=1}^{m} C^*(\Gamma_{i}).
 \end{equation*}
  In particular, $ C^*(\sqcup_{i=1}^{m} ~~ L_{n_i}) \cong \oplus_{i=1}^{m} \mathcal{O}_{n_i}$ and  $C^*(\sqcup_{i=1}^{m} ~~ P_{n_i}) \cong \oplus_{i=1}^{m} M_{n_i+1}$.
\subsubsection{\underline{KMS State on Graph $C^*$-algebra at critical inverse temperature}} 
For a finite directed graph $\Gamma$ without sinks and isolated vertices, the spectral radius of the adjacency matrix of $\Gamma$ is denoted by $\rho(A(\Gamma))$. Let $\mathcal{P}$ be the probability measure on $V(\Gamma)$. 
$\mathcal{F}_{0}= \mathbb{C}1$ and $\mathcal{F}_{i}=span\{S_{\gamma}S_{\mu}^{*}: |\gamma|=|\mu|=i\}$ for $i \geq 1 $ are the finite dimensional algebras such that $\mathcal{F}_{i} \subset \mathcal{F}_{i+1}$. For the proof of the following proposition, see Proposition 4.1 of \cite{Laca} and Proposition 2.21  of \cite{Mandalkms}.   \\
\begin{prop} \label{kmsprop}
For a directed graph $\Gamma=\{ V(\Gamma), E(\Gamma),s,r \}$,\\
\begin{enumerate}
\item The  ${KMS}_{ln {\rho}(A(\Gamma))}$ state exists on $C^*(\Gamma)$ such that 
$${KMS}_{ln {\rho}(A(\Gamma))} (S_{\gamma}S_{\mu}^{*})= \begin{cases}                                    
                                \rho(A(\Gamma))^{-|\gamma|} ~\mathcal{P}_{r{(\gamma})} & if ~ \gamma= \mu \\
                                0 & ~ otherwise  \end{cases} $$ iff $A(\Gamma)\mathcal{P}=\rho(A(\Gamma))\mathcal{P}$. \\
 \item The graph $C^*$-algebra $C^*(\Gamma)$ has a state ${KMS}_{ln {\rho}(A(\Gamma))}$ which is faithful on each $\mathcal{F}_{k}$ iff $\rho(A(\Gamma))$ is an eigenvalue of $A(\Gamma)$ corresponding to an  eigenvector $(\mathcal{P}_{1},\mathcal{P}_{2},...,\mathcal{P}_{|V(\Gamma)|})$ with $ \mathcal{P}_{i} > 0~~\forall i \in [|V(\Gamma)|]$. If the eigenvector $(\mathcal{P}_{1},\mathcal{P}_{2},...,\mathcal{P}_{|V(\Gamma)|})$ is normalized with $ \sum\limits_{i \in [|V(\Gamma)|]} \mathcal{P}_{i} =1  $, then $${KMS}_{ln {\rho}(A(\Gamma))} (S_{\gamma}S_{\mu}^{*})= \begin{cases}                                    
                                \rho(A(\Gamma))^{-|\gamma|} ~\mathcal{P}_{r{(\gamma})} & if ~ \gamma= \mu \\
                                0 & ~ otherwise  \end{cases}. $$
\end{enumerate}
\end{prop}
 
\begin{cor}\label{directkmseq}
For the  Cuntz algebras $\{\mathcal{O}_{n_i}\}_{i=1}^{m}$, $\oplus_{i=1}^{m} ~~\mathcal{O}_{n_i} (\cong C^*(\sqcup_{i=1}^{m} ~L_{n_i}))$ has a ${KMS}_{ln(max\{n_i\}_{i=1}^{m})}$ state. Moreover,
if all $n_i$'s are equal, then $C^*(\sqcup_{i=1}^{m} ~L_{n_i})$ has a ${KMS}_{ln({n_1})}$ state which is faithful on each $\mathcal{F}_{k}$ such that
 $${KMS}_{ln({n_1})} (S_{\gamma}S_{\mu}^{*})= \begin{cases}                                    
                                \frac{{n_1}^{-|\gamma|}}{m} & if ~ \gamma= \mu \\
                                0 & ~ otherwise  \end{cases} $$
     and ${KMS}_{ln\rho(A(\sqcup_{i=1}^{m} ~L_{n_i}))} = \oplus_{i=1}^{m}  {KMS}_{ln \rho(A(L_{n_i}))} $.                          
\end{cor} 
\begin{proof}
	The existence of ${KMS}_{ln(max\{n_i\}_{i=1}^{m})}$ state directly follows from (1) of Proposition $\ref{kmsprop}$.\\
	If all $n_i$'s equal, then $\frac{1}{m}(1,1,...,1)$ is an eigenvector corresponding to the eigenvalue $n_1$. Hence, the faithfulness of state $KMS_{ln(n_1)}$ clearly follows from (2) of Proposition \ref{kmsprop}. \\
	Lastly, we need to show that ${KMS}_{ln\rho(A(\sqcup_{i=1}^{m} ~L_{n_i}))} = \oplus_{i=1}^{m}  {KMS}_{ln \rho(A(L_{n_i}))}.$ Observe that $\{ \gamma : \gamma \text{ is a path on } \sqcup_{i=1}^{m} L_{n_i} \}= \sqcup_{i=1}^{m} \{ \gamma : \gamma \text{ is a path on } L_{n_i}\}$. Using the fact $S_{\gamma}S_{\mu}=S_{\gamma}^*S_{\mu}=S_{\gamma}S_{\mu}^*=0$ for all $\gamma \in L_{n_i}, \mu \in L_{n_j}$ with $i \neq j$ (by Proposition \ref{graphprop}), one can show that $\overline{span} \{S_{\gamma}S_{\mu}^{*} : \gamma, \mu \in E^{< \infty}(\sqcup_{i=1}^{m} L_{n_i}) \} = \oplus_{i=1}^{m} \overline{span} \{S_{\gamma}S_{\mu}^{*} : \gamma, \mu \in E^{< \infty}(L_{n_i})\}$. Hence, by  $(vi)$ of Proposition \ref{graphprop}, the states ${KMS}_{ln\rho(A(\sqcup_{i=1}^{m} ~L_{n_i}))}$ and  $\oplus_{i=1}^{m}  {KMS}_{ln \rho(A(L_{n_i}))}$ have the same domain set. Moreover, for $\gamma, \mu \in L_{n_j}$,
	\begin{align*}
	{KMS}_{ln\rho(A(\sqcup_{i=1}^{m} ~L_{n_i}))}(S_{\gamma}S_{\mu}^{*})=& \begin{cases}                                    
	\frac{\rho(A(\sqcup_{i=1}^{m} L_{n_i}))^{-|\gamma|}}{m} & if ~ \gamma= \mu \\
	0 & ~ otherwise  \end{cases} \\=& \begin{cases}                                    
	\frac{{n_1}^{-|\gamma|}}{m} & if ~ \gamma= \mu \\
	0 & ~ otherwise  \end{cases} =\oplus_{i=1}^{m}  {KMS}_{ln \rho(A(L_{n_i}))}(S_{\gamma}S_{\mu}^{*}).
	\end{align*}
 \end{proof}  
                              
\subsection{Compact quantum groups and quantum automorphism groups } \label{CQG}
In this subsection, we will recall some important facts related to compact quantum groups and their actions on a given $C^{*}$-algebra. We refer the readers to \cite{Van}, \cite{Wang}, \cite{Wor}, \cite{Tim}, \cite{Nesh}, \cite{Fres} for more details.
\begin{Def}
	A compact quantum group (CQG) is a pair $(\mathcal{Q}, \Delta ) $, where $\mathcal{Q}$ is a unital $C^{*}$-algebra and $ \Delta : \mathcal{Q} \to \mathcal{Q} \otimes \mathcal{Q} $  is a unital $C^{*}$-homomorphism such that \vspace{0.1cm}
	\begin{itemize}
		\item[(i)] $(id_{\mathcal{Q}} \otimes \Delta)\Delta = (\Delta \otimes id_{\mathcal{Q}})\Delta $. \vspace{0.1cm}
		\item[(ii)] $ span\{ \Delta(\mathcal{Q})(1 \otimes \mathcal{Q} )\}$ and $ span\{\Delta(\mathcal{Q})(\mathcal{Q} \otimes 1)\} $ are dense in $(\mathcal{Q}\otimes \mathcal{Q})$. 
	\end{itemize}
	
\noindent Given two compact quantum groups, $(\mathcal{Q}_{1},\Delta_{1})$ and $(\mathcal{Q}_{2},\Delta_{2})$, a compact quantum group morphism (CQG morphism) between $\mathcal{Q}_{1}$ and $\mathcal{Q}_{2}$ is a $C^{*}$-homomorphism $ \phi: \mathcal{Q}_{1} \to \mathcal{Q}_{2} $ such that $ (\phi \otimes \phi)\Delta_{1}=\Delta_{2}\phi $ .
\end{Def}
     
For any CQG $\mathcal{Q}$, there exists a canonical dense Hopf $*$-algebra $\mathcal{Q}_{0} \subseteq  \mathcal{Q} $ in which one can define an antipode $\kappa$ and a counit $\epsilon$.\\

But in this article, we are interested in a special class of CQG called compact matrix quantum group.
\begin{Def} \label{CQMG1}
	Let $A$ be a unital $C^*$-algebra and $q=(q_{ij})_{n \times n}$ be a matrix whose entries are coming from $A$. $(A,q)$ is called a compact matrix quantum group (CMQG) if there exists a $*$-subalgebra $A'$ generated by the entries of $q$ such that the following conditions are satisfied:
	\begin{enumerate}
		\item[(i)] $A'$ dense in $A$.
		\item[(ii)] There exists a $C^*$-homomorphism $\Delta: A \to A \otimes A $ such that $\Delta(q_{ij})= \sum_{k=1}^{n}q_{ik} \otimes q_{kj}$ for all $i,j \in [n]$.
		\item[(iii)] There exists a linear anti-multiplicative map $\kappa: A' \to A'$ such that $\kappa(\kappa(a^*))^*=a$ for all $a \in A'$ and $q_{\kappa}:=(\kappa(q_{ij}))_{n \times n}$ is the inverse of $q$. 
	\end{enumerate}
\end{Def}

\noindent The following definition of CQMG is equivalent to Definition \ref{CQMG1} (see \cite{remark} for more details).
\begin{Def} \label{CMQG2}
	Let $A$ be a unital $C^*$-algebra and $q=(q_{ij})_{n \times n}$ be a matrix whose entries are coming from $A$. $(A,q)$ is called a compact matrix quantum group (CMQG) if the following conditions are satisfied:
	\begin{enumerate}
		\item[(i)] A is generated by the entries of matrix $q$,
		\item[(ii)] both $q$ and ${q}^t:=(q_{ji})_{n \times n}$ are invertible in $M_n(A)$,
		\item[(iii)] there exists a $C^*$-homomorphism $\Delta: A \to A \otimes A $ such that $\Delta(q_{ij})= \sum_{k=1}^{n}q_{ik} \otimes q_{kj}$ for all $i,j \in [n]$. 
	\end{enumerate}
	The matrix $q$ is called the fundamental representation of the CMQG $(A,q)$.
\end{Def} 

\begin{Def} \label{identical}
	Let $(A,q)$ and  $(A',q')$ be two CMQGs, where $A$ and $A'$ are unital $C^*$-algebras with fundamental representations $q=(q_{ij})_{n \times n}$ and $q'=(q_{ij}')_{n \times n}$ respectively.
	\begin{enumerate}
	\item $(A,q)$ and $(A',q')$ are said to be { \bf identical} (denoted by $(A,q) \approx (A',q')$) if there exists a $C^*$-isomorphism $ \phi: A \to A' $ such that $ \phi(q_{ij})=q_{ij}'$.
	\item 	$(A',q')$ is said to be a quantum subgroup of $(A,q)$ (denoted by $A' \subset A$) if there exists a surjective $C^*$-homomorphism $ \phi: A \to A' $ such that $ \phi(q_{ij})=q_{ij}'$.\\
	\end{enumerate}
\end{Def}

Next, we will see some examples of CMQGs which will appear in this article.\\

\noindent Examples:
\begin{enumerate}
	\item For $n \in \mathbb{N}$, $C(S_n^+)$ is the universal $C^{*}$-algebra generated by $\{u_{ij}\}_{i,j=1,2,...,n}$ such that  \vspace{0.1cm}
	\begin{itemize}
		\item[(i)]  $u_{ij}^{2}=u_{ij}=u_{ij}^{*} $ for all $i,j \in [n]$, \vspace{0.1cm}
		\item[(ii)]  $ \sum_{k=1}^{n} u_{ik}= \sum_{k=1}^{n} u_{kj}=1 $ for all $i,j \in [n]$. \vspace{0.1cm}
	\end{itemize} 
	Define a coproduct $ \Delta: C(S_n^+) \to C(S_n^+) \otimes C(S_n^+) $ on generators by $\Delta(u_{ij})= \sum_{k=1}^{n}u_{ik} \otimes u_{kj}$. Then $(S_n^+,\Delta)$ [respectively, $(S_n^+,u)$] is called a CQG [respectively, CMQG] whose underlying $C^*$-algebra is $C(S_n^+)$ (see \cite{Wang}, \cite{BBC} for more details).\\
	\item Let  $F \in \mathbb{GL}_{n}\mathbb{(C)}$ and $ A_{U^{t}}(F)$ be the universal $C^{*}$-algebra generated by  $\{q_{ij}: i,j \in [n]\}$ such that the matrix $U=(q_{ij})_{n \times n}$ satisfies the following conditions: \vspace{0.1cm}
	\begin{itemize}
		\item $(U^t)(U^t)^*=(U^t)^*(U^t)=I_{n \times n}$, i.e. $U^{t}$ is unitary. \vspace{0.1cm}
		\item $UF^{-1}U^{*}F=F^{-1}U^{*}FU=I_{n \times n}$. \vspace{0.1cm} 
	\end{itemize}
	Again, coproduct on generators $\{q_{ij} : {i,j \in [n]}\}$ is given by $\Delta(q_{ij})=\sum_{k=1}^{n} q_{ik}\otimes q_{kj} $. It can be shown that $(A_{U^{t}}(F),\Delta)$ is a CQG (as well as a CMQG with respect to the fundamental representation $U$).\\
	If $ F=I_{n \times n} $, then we write $ A_{U^{t}}(F) $ as $U_{n}^{+}$, i.e.
	  $(U_{n}^{+},\Delta):=(A_{U^{t}}(I_{n \times n}),\Delta)$  (consult \cite{Van} for details).\\

	
	
	\item  $C(H_{n}^{\infty +})$ is defined to be the universal $C^*$-algebra generated by $\{u_{ij}: i,j \in [n]\}$ such that \vspace{0.1cm}
	\begin{itemize}
		\item the matrices $u=(u_{ij})_{n \times n}$ and $(u_{ij}^{*})_{n \times n}$ are unitaries, \vspace{0.1cm}
		\item $u_{ij}$'s are normal partial isometries for all $i,j$. \vspace{0.1cm}
	\end{itemize}
	The coproduct $\Delta $ on generators is again given by $\Delta(u_{ij})=\sum_{k=1}^{n} u_{ik}\otimes u_{kj}$. Then $(C(H_{n}^{\infty +}), \Delta)$ forms a CQG. Moreover, the CMQG $(C(H_{n}^{\infty +}), u)$ forms a unitary easy quantum group with respect to the category of partition $\mathcal{C}_0 =\{ \begin{tikzpicture}[scale=0.5]
		\draw[fill=black] (0,0) circle (3pt);
		\draw[fill=black] (0,1) circle (3pt);
		\draw (0.6,0) circle (3pt);
		\draw (0.6,1) circle (3pt);
		\draw[black] (0,0.05)--(0,0.3);
		\draw[black] (0,0.95)--(0,0.7);
		\draw[black] (0.6,0.05)--(0.6,0.3);
		\draw[black] (0.6,0.95)--(0.6,0.7);
		\draw[black] (0,0.3)--(0.6,0.3);
		\draw[black] (0,0.7)--(0.6,0.7);
		\draw[black] (0.3,0.3) -- (0.3,0.7);\end{tikzpicture}, 
	\begin{tikzpicture}[scale=0.5]
		\draw (0,0) circle (3pt);
		\draw (0,1) circle (3pt);
		\draw[fill=black] (0.6,0) circle (3pt);
		\draw[fill=black] (0.6,1) circle (3pt);
		\draw[black] (0,0.05)--(0,0.3);
		\draw[black] (0,0.95)--(0,0.7);
		\draw[black] (0.6,0.05)--(0.6,0.3);
		\draw[black] (0.6,0.95)--(0.6,0.7);
		\draw[black] (0,0.3)--(0.6,0.3);
		\draw[black] (0,0.7)--(0.6,0.7);
		\draw[black] (0.3,0.3) -- (0.3,0.7);\end{tikzpicture}  , \begin{tikzpicture}[scale=0.5]
		\draw (0,0) circle (3pt);
		\draw (0,1) circle (3pt);
		\draw (0.6,0) circle (3pt);
		\draw (0.6,1) circle (3pt);
		\draw[black] (0,0.05)--(0,0.3);
		\draw[black] (0,0.95)--(0,0.7);
		\draw[black] (0.6,0.05)--(0.6,0.3);
		\draw[black] (0.6,0.95)--(0.6,0.7);
		\draw[black] (0,0.3)--(0.6,0.3);
		\draw[black] (0,0.7)--(0.6,0.7);
		\draw[black] (0.3,0.3) -- (0.3,0.7);\end{tikzpicture} \}$  (see \cite{easy} for the details on the unitary easy quantum group). We denote this CQG (or CMQG) by $H_{n}^{\infty +}$.\\
	
	\item  Consider the universal $C^*$-algebra generated by $\{u_{ij}: i,j \in [n]\}$ such that \vspace{0.1cm}
	\begin{enumerate}
		\item[(i)] the matrices $(u_{ij})_{n \times n}$ and $(u_{ij}^{*})_{n \times n}$ are unitaries, \vspace{0.1cm}
		\item[(ii)] $u_{ij}$'s are partial isometries for all $i,j$. \vspace{0.1cm}
	\end{enumerate}
	and denote it by  $C(SH_{n}^{\infty +})$.
	Observe that condition (ii) can be replaced by the condition $`` u_{ik}u_{jk}^*=u_{ik}^*u_{jk}=0 ~~ \forall i,j,k \in [n] \text{ with } i \neq j"$.
	Similarly, the coproduct $\Delta $ on generators is given by $\Delta(u_{ij})=\sum_{k=1}^{n} u_{ik}\otimes u_{kj}$.
	 Thus, the CMQG $(C(SH_{n}^{\infty +}),u)$ constitutes a unitary easy quantum group with respect to the category of partition $\mathcal{C}_0=\{\begin{tikzpicture}[scale=0.5]
		\draw[fill=black] (0,0) circle (3pt);
		\draw[fill=black] (0,1) circle (3pt);
		\draw (0.6,0) circle (3pt);
		\draw (0.6,1) circle (3pt);
		\draw[black] (0,0.05)--(0,0.3);
		\draw[black] (0,0.95)--(0,0.7);
		\draw[black] (0.6,0.05)--(0.6,0.3);
		\draw[black] (0.6,0.95)--(0.6,0.7);
		\draw[black] (0,0.3)--(0.6,0.3);
		\draw[black] (0,0.7)--(0.6,0.7);
		\draw[black] (0.3,0.3) -- (0.3,0.7);\end{tikzpicture}, 
	\begin{tikzpicture}[scale=0.5]
		\draw (0,0) circle (3pt);
		\draw (0,1) circle (3pt);
		\draw[fill=black] (0.6,0) circle (3pt);
		\draw[fill=black] (0.6,1) circle (3pt);
		\draw[black] (0,0.05)--(0,0.3);
		\draw[black] (0,0.95)--(0,0.7);
		\draw[black] (0.6,0.05)--(0.6,0.3);
		\draw[black] (0.6,0.95)--(0.6,0.7);
		\draw[black] (0,0.3)--(0.6,0.3);
		\draw[black] (0,0.7)--(0.6,0.7);
		\draw[black] (0.3,0.3) -- (0.3,0.7);\end{tikzpicture} \}$ (see section 4 of \cite{easy} for details on the unitary easy quantum group).  We denote this CQG (or CMQG) by $SH_{n}^{\infty +}$.\\
	
	\item Let $ n \in \mathbb{N} $ and $(Q, \Delta_{Q})$ be a CQG. The free wreath product of $Q$ by the quantum permutation group $S_n^+$, denoted by $Q \wr_{*} S_n^+$, is the quotient of the algebra $\underbrace{Q*Q*\cdots*Q}_{n-times}*S_n^+$ by the ideals of the form $(i_k(a)t_{kl}-t_{kl}i_k(a))$ for $ k,l \in [n], a \in Q $ and $ i_k:Q \to \underbrace{Q*Q*\cdots*Q}_{n-times}*S_n^+ $ is the natural inclusion. Define the coproduct $\Delta : Q \wr_{*} S_n^+ \to (Q \wr_{*} S_n^+) \otimes (Q \wr_{*} S_n^+)$ by $\Delta(i_k(a))=\sum\limits_{l=1}^{n} i_k \otimes i_l(\Delta_{Q}(a))(t_{kl} \otimes 1)$ and $\Delta(t_{kl})=\sum\limits_{r=1}^{n} t_{kr} \otimes t_{rl}$. Moreover, the counit satisfies $\epsilon(t_{kl})=\delta_{kl}$ and $ \epsilon(i_k(a))=\epsilon_{Q}(a)$, where $\epsilon_{Q}$ is the counit of $(Q, \Delta_{Q})$ (see \cite{Wangfree}, \cite{Bichonwreath}).\\
	
\end{enumerate}

Next, we will discuss the CQG action and the quantum symmetry of a $C^*$-algebra from the categorical viewpoint. The readers are referred to \cite{Wang} and \cite{Bichon} for the following definitions and discussions.
\begin{Def}
A CQG $(\mathcal{Q},\Delta)$ is said to be acting faithfully on a unital $C^{*}$-algebra $\mathcal{C}$ if there exists a unital $C^{*}$-homomorphism $\alpha:\mathcal{C} \to \mathcal{C} \otimes \mathcal{Q} $ such that: \vspace{0.1cm}
\begin{itemize}
		\item[(i)] Action equation: $(\alpha \otimes id_{\mathcal{Q}})\alpha = (id_{\mathcal{C}} \otimes \Delta)\alpha.$ \vspace{0.1cm}
		\item[(ii)] Podle\'s condition: $span\{\alpha(\mathcal{C})(1\otimes \mathcal{Q})\}$ is dense in $\mathcal{C} \otimes \mathcal{Q}.$ \vspace{0.1cm}
		\item[(iii)] Faithfulness: The $*$-algebra generated by the set $\{(\theta\otimes id)\alpha(\mathcal{C}) : \theta \in \mathcal{C}^{*}\}$ is norm-dense in $\mathcal{Q}$.\vspace{0.1cm}
\end{itemize}
\end{Def}
$((\mathcal{Q},\Delta),\alpha)$ is also called a \textbf{quantum transformation group of $\mathcal{C} $}.\\

\noindent Given a unital $C^{*}$-algebra $ \mathcal{C} $, the \textbf{category of quantum transformation groups of $\mathcal{C}$} is a \textbf{category} $ \mathfrak{C} $ containing quantum transformation groups of $\mathcal{C}$ as objects and a morphism from $((\mathcal{Q}_{1},\Delta_{1}),\alpha_{1})$ to $((\mathcal{Q}_{2},\Delta_{2}),\alpha_{2})$  be a CQG morphism $\phi :(\mathcal{Q}_{1},\Delta_{1}) \to (\mathcal{Q}_{2},\Delta_{2})$ such that $ (id_{\mathcal{C}} \otimes \phi)\alpha_{1}=\alpha_{2} $.\\
The \textbf{universal object of the category $ \mathfrak{C} $} is a quantum transformation group of $\mathcal{C}$, denoted by $((\widehat{\mathcal{Q}},\widehat{\Delta}),\widehat{\alpha})$, satisfying the following universal property :\\
For any object $((\mathcal{D}, \Delta_{\mathcal{D}}),\delta)$ from the category of quantum transformation groups of $\mathcal{C}$, there is a surjective CQG morphism $\widehat{\phi}: (\widehat{\mathcal{Q}}, \widehat{\Delta}) \to (\mathcal{D}, \Delta_{\mathcal{D}}) $ such that $(id_{\mathcal{C}} \otimes \widehat{\phi})\widehat{\alpha}=\delta $. 

\begin{Def}
	Given a unital $C^{*}$-algebra $\mathcal{C}$, the { \bf quantum automorphism group } of $ \mathcal{C} $ is the underlying CQG of the universal object of the category of quantum transformation group of $\mathcal{C}$ if the universal object exists.
\end{Def}

\begin{rem}
	In the above category, the universal object might fail to exist in general. One can recover from this situation by restricting the category to a subcategory so that a universal object would exist.	
	Take a linear functional $ \tau: \mathcal{C} \to \mathbb{C} $. Define a subcategory $\mathfrak{C}_{\tau}$ whose objects are those quantum transformation groups of $\mathcal{C}$, $((\mathcal{Q},\Delta),\alpha)$ for which $(\tau \otimes id)\alpha(.)=\tau(.).1 $ on a suitable subspace of $\mathcal{C}$ and morphisms are taken as the above. 
\end{rem}

\noindent Examples:
\begin{enumerate}
	\item For $n$ points space $X_{n}$, the universal object in the category of quantum transformation group of $C(X_{n})$ exists and is isomorphic to the quantum permutation group, $S_{n}^{+}$  (see \cite{Wang}, \cite{BBC} for more details).\\
	
	\item For the $C^{*}$-algebra $ M_{n}(\mathbb{C}) $, the universal object in the category of quantum transformation group of $ M_{n}(\mathbb{C}) $ (for $n \geq 2 $) does not exist. But if we fix a linear functional $\tau'$ on $ M_{n}(\mathbb{C}) $ which is defined by $\tau'(A)=Tr(A)$ and assume that any object of the category also preserves $\tau'$, i.e. $(\tau' \otimes id)\alpha(.)=\tau'(.).1 $ on $M_{n}(\mathbb{C})$, then the universal object exists in $\mathfrak{C}_{\tau'}$ (see \cite{Wang} for more details).
\end{enumerate}

\subsection{Quantum symmetry in an orthogonal filtration preserving way}
In this subsection, we recall the quantum symmetry of a $C^*$-algebra equipped with an orthogonal filtration with respect to a given state $\phi$ on that $C^*$-algebra introduced by Banica and Skalski in \cite{Ortho}. 

\begin{Def}
	Let $\mathcal{C}$ be a unital $C^{\ast}$-algebra together with a faithful state $\phi$ and a family $\{F_{i}\}_{i\in \mathcal{I}}$ of finite dimensional subspaces of $\mathcal{C}$ (where  $\mathcal{I}$ is the index set containing a distinguished element $0$). The collection $(\{F_{i}\}_{i\in \mathcal{I}}, \phi)$ defines an orthogonal filtration on $\mathcal{C}$ if the following conditions are satisfied:
	\begin{enumerate}
		\item $F_{0}=\mathbb{C}1_{\mathcal{C}}$;
		\item If $a\in F_{i},b\in F_{j}$ for $i,j \in \mathcal{I}$ with $i \neq j $, then $\phi(a^{*}b)=0$;
		\item The set $ span(\cup_{i\in \mathcal{I}} F_{i})$ is dense in $\mathcal{C}$.
	\end{enumerate}
\end{Def}
\noindent {Examples:}
\begin{enumerate}
	\item Any unital separable AF algebra admits an orthogonal filtration.\\
	
	\item  In Section 5 of \cite{Mandalkms}, the authors provide an orthogonal filtration for Cuntz algebra $\mathcal{O}_{n}$ (viewing it as a graph $C^*$-algebra with respect to the graph $ L_{n}$). We will similarly define an orthogonal filtration on the direct sum of Cuntz algebras $\oplus_{i=1}^{m} \mathcal{O}_{n_i}$ (whose underlying graph is $\sqcup_{i=1}^{m} L_{n_i}$). Let us define $\mathcal{F}_{0}= \mathbb{C}1$ and $\mathcal{F}_{i}=span\{S_{\gamma}S_{\mu}^{*}: |\gamma|=|\mu|=i\}$ for $i \geq 1 $. Note that $\mathcal{F}_{i} \subset \mathcal{F}_{i+1} $. Consider the KMS state $KMS_{ln(n_i)}$ on $C^*(L_{n_i})$ and $\oplus_{i=1}^{m} KMS_{ln(n_i)}$ on $C^*(\sqcup_{i=1}^{m} L_{n_i})$. Now, define finite dimensional vector subspaces: $\mathcal{W}_0=\mathcal{F}_{0}= \mathbb{C}1$ and 
	$\mathcal{W}_i=\mathcal{F}_{i} \ominus \mathcal{F}_{i-1}$, the orthogonal complement of $\mathcal{F}_{i-1}$ in $\mathcal{F}_{i}$ for $i \geq 1$. Moreover, we define the finite-dimensional subspaces
	$$ \mathcal{M}_{k,l}^{(1)}:= span\{S_{\mu}x: |\mu|=l,x \in \mathcal{W}_{k}\}$$ and $$ \mathcal{M}_{k,l}^{(2)}:= span\{ yS_{\nu}^*: |\nu|=l, y \in \mathcal{W}_{k}\}$$ for $(k,l) \in \mathbb{N}_{0} \times \mathbb{N}$.\\
	Now, the collection $(\{ \mathcal{W}_{i}, \mathcal{M}_{p,q}^{(1)}, \mathcal{M}_{r,s}^{(2)} : i \in \mathbb{N}_{0}, (p,q) \in \mathbb{N}_{0} \times \mathbb{N}, (r,s) \in \mathbb{N}_{0} \times \mathbb{N} \}, \oplus_{i=1}^{m} KMS_{ln(n_i)})$ is an orthogonal filtration on $\oplus_{i=1}^{m} \mathcal{O}_{n_i}$.\\
	For the proof, we would like to mention a key fact: if we take  any two paths $\gamma$ and  $\mu$ from $L_k$ and $L_l$ respectively (for $k \neq l$), then $S_{\gamma} S_{\mu}=S_{\gamma}^* S_{\mu}=S_{\gamma}S_{\mu}^*=0$ (by Proposition \ref{graphprop}). 
	Now, the arguments can be adopted from Theorem 5.3 of \cite{Mandalkms}. Though one can notice that Lemma 5.1 of \cite{Mandalkms} does not hold for $\oplus_{i=1}^{m} \mathcal{O}_{n}$ with respect to the state $\oplus_{i=1}^{m} KMS_{ln(n_i)}$, it's true for each individual $\mathcal{O}_{n_i}$ with the state $ KMS_{ln(n_i)}$. Hence, the fact mentioned above ensures that similar proof works in a trouble-free way. \\ 
		 
\end{enumerate}

\begin{Def} \label{orthofilpreway}	
	For an orthogonal filtration $ \mathfrak{F}= (\{F_{i}\}_{i\in \mathcal{I}}, \phi)$ on a unital $C^*$-algebra $\mathcal{C}$, we say that  a CQG $(\mathcal{Q},\Delta)$ acts on $\mathcal{C}$ by $\alpha$ in a filtration preserving way if  
	$$	\alpha(F_{i})\in F_{i}\otimes  \mathcal{Q} \ \  \text{for all} \ \  i\in \mathcal{I} . $$	
\end{Def}

\begin{Def}
	 We define a category $\mathfrak{C}_{\mathfrak{F}}(\mathcal{C})$ whose objects are $((\mathcal{Q},\Delta), \alpha)$ such that $(\mathcal{Q},\Delta)$ acts on $\mathcal{C}$ in an orthogonal filtration preserving way in the sense of Definition \ref{orthofilpreway} and a morphism from $((\mathcal{Q}_{1},\Delta_{1}),\alpha_{1})$ to $((\mathcal{Q}_{2},\Delta_{2}),\alpha_{2})$ is given by a CQG morphism $\Phi:\mathcal{Q}_{1} \to \mathcal{Q}_{2} $ such that $ (id_{\mathcal{C}} \otimes \phi)\alpha_{1}=\alpha_{2} $.	
\end{Def}

\noindent For the proof of the theorem below, consult Theorem 2.7 from \cite{Ortho}.
\begin{thm}
	 For an orthogonal filtration $\mathfrak{F}=(\{F_{i}\}_{i\in \mathcal{I}}, \phi)$ on a unital $C^*$-algebra $\mathcal{C}$, there exists a universal object in the category $\mathfrak{C}_{\mathfrak{F}}(\mathcal{C})$.
\end{thm}	

\noindent We denote the underlying CQG of the universal object by $Q_{\mathfrak{F}}(\mathcal{C})$ in the category $\mathfrak{C}_{\mathfrak{F}}(\mathcal{C})$.

\subsection{Quantum symmetry of a graph $C^{*}$-algebra} \label{QSgraph}
Let $\Gamma=\{V(\Gamma), E(\Gamma), s, r\}$ be a finite directed graph without isolated vertices. Therefore, to define a CQG action on $C^*(\Gamma)$, it is enough to define an action on the partial isometries corresponding to the edges of $\Gamma$. To describe the quantum symmetries of a graph $C^*$-algebra, { \bf we restrict ourselves to a finite directed graph without isolated vertices} and we simply call it a graph.   
\subsubsection{\underline{Quantum symmetry of a graph $C^{*}$-algebra in the category $\mathfrak{C}_{\tau}^{Lin}$ introduced in \cite{Mandal}}}
\begin{Def}(Definition 3.4 of \cite{Mandal})
	Given a graph $\Gamma$, a faithful action $\alpha$ of a CQG $\mathcal{Q}$ on a $C^{*}$-algebra $C^{*}(\Gamma)$ is said to be linear if $ \alpha(S_{e})=\sum_{f \in E(\Gamma)} S_{f} \otimes q_{fe}$, where $q_{ef}\in \mathcal{Q}$ for each $e,f \in E(\Gamma)$.
\end{Def}

\noindent Take
\begin{itemize}
	\item[(1)] $\mathcal{V}=\{u\in V(\Gamma) : u \text{ is not a source of any edge of } \Gamma\}$ \vspace{0.1cm}.
	\item[(2)] $ \mathcal{E}= \{(e,f) \in E(\Gamma) \times E(\Gamma) : S_{e}S_{f}^{*} \neq 0 \}=\{(e,f) \in E(\Gamma) \times E(\Gamma) : r(e)=r(f) \}$. \vspace{0.1cm}
\end{itemize}
In Lemma 3.2 of \cite{Mandal}, it was shown that $\{ p_{u}, S_{e}S_{f}^{*} : u \in \mathcal{V}, (e,f) \in \mathcal{E} \}$ is a linearly independent set. \\[0.2cm]
Now, define $\mathcal{V}_{2,+}= span\{ p_{u}, S_{e}S_{f}^{*} : u \in \mathcal{V}, (e,f) \in \mathcal{E} \}$ and a linear functional $\tau: \mathcal{V}_{2,+} \to \mathbb{C}$ by $\tau(S_{e}S_{f}^{*})=\delta_{ef}$, $\tau(p_{u})=1 $ for all $ (e,f) \in \mathcal{E} $ and $ u \in \mathcal{V}$
(see subsection 3.1 of \cite{Mandal}).\\[0.1cm]

Since $\alpha(\mathcal{V}_{2,+}) \subseteq \mathcal{V}_{2,+} \otimes \mathcal{Q}$ by Lemma 3.6 of \cite{Mandal}, the equation $(\tau \otimes id)\alpha(.)=\tau(.).1 $ on $\mathcal{V}_{2,+}$ makes sense.\\

\begin{Def}(Definition 3.7 of \cite{Mandal})
	For a graph $\Gamma $, define a category $\mathfrak{C}_{\tau}^{Lin} $ whose objects are $((\mathcal{Q},\Delta),\alpha) $, quantum transformation group of $ C^{*}(\Gamma)$ such that $(\tau \otimes id)\alpha(.)=\tau(.).1 $ on $\mathcal{V}_{2,+}$. A morphism from $((\mathcal{Q}_{1},\Delta_{1}),\alpha_{1})$ to $((\mathcal{Q}_{2},\Delta_{2}),\alpha_{2})$ is given by a CQG morphism $\Phi:\mathcal{Q}_{1} \to \mathcal{Q}_{2} $ such that $ (id_{C^{*}(\Gamma)} \otimes \phi)\alpha_{1}=\alpha_{2} $.
\end{Def}

$F^{\Gamma}$ is a $ (|E(\Gamma)| \times |E(\Gamma)| ) $ matrix such that $(F^{\Gamma})_{ef}= \tau(S_{e}^{*}S_{f})$. It can be shown that $F^{\Gamma}$ is an invertible diagonal matrix. Therefore, $A_{U^{t}}(F^{\Gamma})$ is a CQG. We refer to Proposition 3.8 and Theorem 3.9 of \cite{Mandal} for the proof of the following theorem.          
\begin{thm}
	For a graph $\Gamma$,\vspace{0.1cm}
	\begin{enumerate}
		\item there is a surjective $C^{*}$-homomorphism from $A_{U^{t}}(F^{\Gamma})$ to any object in category $\mathfrak{C}_{\tau}^{Lin}$. \\
		\item the category $\mathfrak{C}_{\tau}^{Lin} $ admits a universal object.\\
	\end{enumerate}
	We denote the underlying CQG and the respective action of the universal object by $ Q_{\tau}^{Lin}(\Gamma) $ and $\alpha$ respectively, with respect to category $\mathfrak{C}_{\tau}^{Lin} $. Note from the proof of Theorem 3.9 of \cite{Mandal} that $ Q_{\tau}^{Lin}(\Gamma) $ is essentially a CMQG with the fundamental representation $q=(q_{ef})_{|E(\Gamma)|\times |E(\Gamma)|}$ such that $\alpha(S_e)=\sum\limits_{f \in E(\Gamma)} S_f \otimes q_{fe}$. \\
\end{thm}

\subsubsection{\underline{Quantum symmetry of a graph $C^{*}$-algebra in the KMS state category introduced in \cite{Mandalkms}}}
\begin{Def}(Definition 3.1 of \cite{Mandalkms})
For a graph $\Gamma$ without a sink, the category $\mathfrak{C}_{KMS}^{Lin}$ whose objects are $((Q,\Delta),\alpha)$, where $(Q,\Delta)$ is a CQG and $\alpha $ is a linear action on $C^*(\Gamma)$ which preserves the state ${KMS}_{ln {\rho}(A(\Gamma))}$, i.e. \\
$$ \alpha(S_{e})=\sum_{f \in E(\Gamma)} S_{f} \otimes q_{fe} $$ and on $C^*(\Gamma)$
$$ ({KMS}_{ln {\rho}(A(\Gamma))} \otimes id_{C^*(\Gamma)} ) \circ \alpha (.) =  {KMS}_{ln {\rho}(A(\Gamma))}(.)1 .$$
A morphism from $((\mathcal{Q}_{1},\Delta_{1}),\alpha_{1})$ to $((\mathcal{Q}_{2},\Delta_{2}),\alpha_{2})$ is given by a CQG morphism $\Phi:\mathcal{Q}_{1} \to \mathcal{Q}_{2} $ such that $ (id_{C^{*}(\Gamma)} \otimes \Phi)\alpha_{1}=\alpha_{2} $.
\end{Def}

\begin{thm}(Proposition 3.2 of \cite{Mandalkms} ) \label{kmsthm1}
The category $\mathfrak{C}_{KMS}^{Lin}$ has a universal object.
\end{thm}
\noindent We denote the underlying CMQG of the universal object by $ Q_{KMS}^{Lin}(\Gamma) $ in category $\mathfrak{C}_{KMS}^{Lin} $.\\

\noindent The next result provides a sufficient condition for the isomorphism between the categories  $\mathfrak{C}_{\tau}^{Lin} $ and $\mathfrak{C}_{KMS}^{Lin} $.
\begin{thm}(Theorem 3.6 of \cite{Mandalkms})\label{kmsthm2}
For a graph $\Gamma$ without a sink, if all the row sums of $A(\Gamma)$ are all $\rho(A(\Gamma))$, then the ${KMS}_{ln {\rho}(A(\Gamma))}$ state exists on $C^*(\Gamma)$ such that 
 $${KMS}_{ln {\rho}(A(\Gamma))} (S_{\gamma}S_{\mu}^{*})= \begin{cases}                                    
                                \frac{{\rho}(A(\Gamma))^{-|\gamma|}}{|V(\Gamma)|} & if ~ \gamma= \mu \\
                                0 & ~ otherwise  \end{cases} $$
and in this case, the categories $\mathfrak{C}_{\tau}^{Lin} $ and $\mathfrak{C}_{KMS}^{Lin} $ coincide.
\end{thm}

\begin{cor}\label{kmscor1}
For the  Cuntz algebras $\{\mathcal{O}_{n_i}\}_{i=1}^{m}$ where all $n_i$'s are equal, the categories $\mathfrak{C}_{\tau}^{Lin} $ and $\mathfrak{C}_{KMS}^{Lin} $ coincide for $\oplus_{i=1}^{m} \mathcal{O}_{n_i} ~~(\cong C^*(\sqcup_{i=1}^{m} ~L_{n_i}))$.\\
\end{cor}

\subsubsection{\underline{Quantum symmetry of direct sum of Cuntz algebras in a natural state preserving category }}
We will adopt the same idea to define a new state preserving category only for the direct sum of Cuntz algebras.
Since each $C^*(L_{n_i})$ has a natural $KMS_{ln(n_i)}$ state, we can define a state $ \oplus_{i=1}^{m} KMS_{ln(n_i)}$ on $C^*(\sqcup_{i=1}^{m} ~~ L_{n_i})$.

\begin{Def}
For $\oplus_{i=1}^{m} \mathcal{O}_{n_i} (\cong C^*(\sqcup_{i=1}^{m} ~L_{n_i}))$, the category $\mathfrak{C}_{\oplus KMS}^{Lin}$ whose objects are $((Q,\Delta),\alpha)$, where $(Q,\Delta)$ is a CQG and $\alpha $ is a linear action on $C^*(\sqcup_{i=1}^{m} ~L_{n_i})$ which preserves $ \oplus_{i=1}^{m} KMS_{ln(n_i)}$, i.e. \\
$$ \alpha(S_{e})=\sum_{f \in E(\sqcup_{i=1}^{m} L_{n_i})} S_{f} \otimes q_{fe} $$ and on $C^*(\sqcup_{i=1}^{m} ~~ L_{n_i})$
$$ ( \oplus_{i=1}^{m} KMS_{ln(n_i)} \otimes id_{C^*(\Gamma)} ) \circ \alpha (.) =  \oplus_{i=1}^{m} KMS_{ln(n_i)}(.)1 .$$
A morphism from $((\mathcal{Q}_{1},\Delta_{1}),\alpha_{1})$ to $((\mathcal{Q}_{2},\Delta_{2}),\alpha_{2})$ is given by a CQG morphism $\Phi:\mathcal{Q}_{1} \to \mathcal{Q}_{2} $ such that $ (id_{C^{*}(\sqcup_{i=1}^{m} ~L_{n_i})} \otimes \Phi)\alpha_{1}=\alpha_{2} $.		
\end{Def}

\begin{lem} \label{qsubgroupAUF}
	There exists a surjective $C^*$-homomorphism from $A_{U}(F^{\Gamma})$ to any object of the category $\mathfrak{C}_{\oplus KMS}^{Lin}$ corresponding to the graph $\Gamma=\sqcup_{i=1}^{m} ~~ L_{n_i}$. 
\end{lem}
\begin{proof}
	For our convenience, we denote the state $KMS_{ln(n_i)}$ simply by $\phi_{i}$ and   $\oplus_{i=1}^{m} KMS_{ln(n_i)}= \oplus_{i=1}^{m} \phi_{i}$ by $\phi$. Also, we define $E_i := E(L_{n_i})$ and $E:=\sqcup_{i=1}^{m} E_i = E(\sqcup_{i=1}^{m} L_{n_i})$. Hence, $|E_i|=n_i$ and $|E|=\sum\limits_{i=1}^{m} n_i =: n$.\\
	Let $((Q,\Delta),\alpha)$ be an object of $\mathfrak{C}_{\oplus KMS}^{Lin}$ with $\alpha(S_{e})=\sum\limits_{f \in E} S_{f} \otimes q_{fe}$, where $q_{fe} \in Q$ for $e,f \in E$. Since $\alpha$ preserves $\phi$, for each $p \in [m]$ and for all $e,f \in E_p$ we have
	\begin{align*}
		 & (\phi \otimes id_{C^*(\Gamma)} ) \circ \alpha (S_e S_f^*) = \phi(S_e S_f^*)1\\
		 \Rightarrow & \sum\limits_{g,h \in E} \phi(S_g S_{h}^{*})~ q_{ge} q_{hf}^{*} = \frac{1}{m} \phi_{i}(S_e S_f^*)=\frac{1}{m.n_p} \delta_{ef} \\
		 \Rightarrow & \sum\limits_{i \in [m]} \sum\limits_{g,h \in E_i} \frac{1}{m} \phi_i(S_g S_{h}^{*})~ q_{ge} q_{hf}^{*} = \frac{1}{m} \phi_{i}(S_e S_f^*)=\frac{1}{m.n_p} \delta_{ef} \\
		 \Rightarrow & \sum\limits_{i \in [m]} \sum\limits_{g,h \in E_i} \frac{1}{m.n_i} \delta_{gh}~ q_{ge} q_{hf}^{*}= \frac{1}{m.n_p} \delta_{ef}\\
		 \Rightarrow & \sum\limits_{i \in [m]} \sum\limits_{g,h \in E_i} \frac{1}{n_i} \delta_{gh}~ q_{ge} q_{hf}^{*}= \frac{1}{n_p} \delta_{ef}
		 \end{align*} 
Therefore, the last equation implies $U^t (F^{\Gamma})^{-1} {U^t}^*= {F^{\Gamma}}^{-1}$, where $U:=(q_{ef})_{n \times n}$.	Since the CQG action is in our consideration, the invertibility of $U^t$ guarantees ${F^{\Gamma}}^{-1} {U^t}^* F^{\Gamma} U^t=I $. Hence, ${U^t}^* F^{\Gamma} U^t=F^{\Gamma}$.\\
Now, observe that since $S_e^* S_e= \sum\limits_{f \in E_p} S_f S_f^*$ for all $e \in E_p$ and $p \in [m]$,
 $$ \phi(S_e^* S_e)= \sum\limits_{f \in E_p} \frac{1}{m}\phi_{p}(S_f S_f^*)= \sum\limits_{f \in E_p} \frac{1}{m}\frac{1}{n_p} = \frac{1}{m}. $$  
  Again using the fact that $\alpha$ preserves $\phi$, for all $e \in E_p$ we get
  \begin{align*}
  	& (\phi \otimes id_{C^*(\Gamma)} ) \circ \alpha (S_e^* S_e) = \phi(S_e^* S_e)1\\
  	\Rightarrow & \sum\limits_{k \in E}\phi(S_k^* S_k)~ q_{ke}^* q_{ke} = \frac{1}{m} \\
  	\Rightarrow & \sum\limits_{i \in [m]} \sum\limits_{k \in E_i} \frac{1}{m}~ q_{ke}^* q_{ke} = \frac{1}{m} \\
  	\Rightarrow & \sum\limits_{i \in [m]} \sum\limits_{k \in E_i}  q_{ke}^* q_{ke} = 1.
  \end{align*}
  Therefore, $U^*U =I_n$. Being a CQG action, invertibility of $U$ again concludes $UU^* =I_n$. Then, the universal property of $A_{U}(F^{\Gamma})$ ensures the statement of the lemma.
\end{proof}
\noindent Now, using similar arguments as in Theorem 3.9 of \cite{Mandal}, one can easily show the following theorem.
\begin{thm} \label{KMS}
	The category $\mathfrak{C}_{\oplus KMS}^{Lin}$ always admits a universal object.
\end{thm}	

\noindent For the graph $\sqcup_{i=1}^{m} ~~ L_{n_i}$, we denote the underlying CQG of the universal object of $\mathfrak{C}_{\oplus KMS}^{Lin}$ by $Q_{\oplus KMS}^{Lin}(\sqcup_{i=1}^{m} ~~ L_{n_i})$. From the proof of the theorem above, one can notice that the universal object $Q_{\oplus KMS}^{Lin}(\sqcup_{i=1}^{m} ~~ L_{n_i})$ appears as a CMQG with the fundamental representation $U:=(q_{ef})_{n \times n}$ (where $n=\sum\limits_{i=1}^{m} n_i $) such that $\alpha(S_e)=\sum\limits_{f \in E(\sqcup_{i=1}^{m} ~~ L_{n_i})} S_f \otimes q_{fe}$. Hence, as a particular case of Lemma \ref{qsubgroupAUF}, we have
\begin{align} 
	& {U^t}^* F^{\Gamma} U^t=F^{\Gamma} \nonumber \\
	\Leftrightarrow & \sum\limits_{i=1}^{m} \sum\limits_{g \in E(L_{n_i})} n_i ~ q_{eg}^* q_{fg} =n_p \delta_{ef} ~~ \text{ for all } e,f \in E(L_{n_p}) \text{ and for each } p \in [m]. \label{relUtF}
\end{align}\\[0.6cm]

\noindent We will compute the quantum symmetry of the direct sum of Cuntz algebras (viewing it as a graph $C^*$-algebra) with respect to the categories introduced in \cite{Mandal} and \cite{Mandalkms}.
\section{Direct sum of non-isomorphic Cuntz algebras}
\begin{thm} 
\label{thm1}
Let $\{\mathcal{O}_{n_i}\}_{i=1}^{m}$ be a finite family of Cuntz algebras where all $n_i$'s are distinct. Then $ Q_{\tau}^{Lin}(\sqcup_{i=1}^{m} L_{n_i} ) \cong *_{i=1}^{m}Q_{\tau}^{Lin}(L_{n_i}) \cong *_{i=1}^{m} U_{n_i}^{+} $.  
\end{thm}

\begin{center}
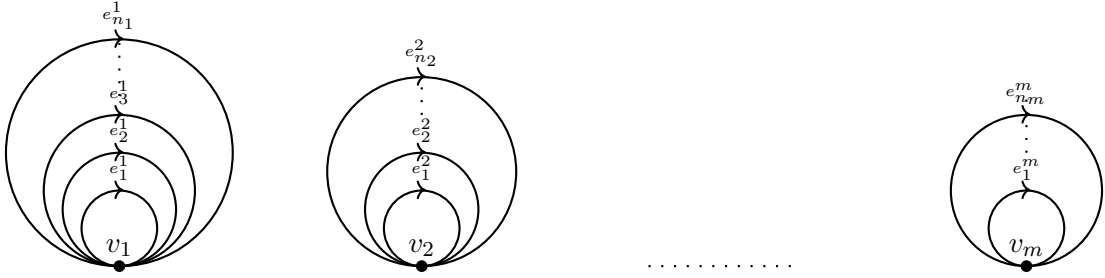
\begin{figure}[htpb]
\begin{tikzpicture}
\draw[fill=black] (0,0) circle (2pt) node[anchor=south]{$v_{1}$};
\draw[black,thick](0,0.5) circle (0.5) node[above=0.25]{\rmidarrow}node[above=0.5]{${\scriptscriptstyle{e_1^1}}$};
\draw[black,thick](0,0.75) circle (0.75) node[above=0.5]{\rmidarrow} node[above=0.75]{${\scriptscriptstyle{e_2^1}}$};
\draw[black,thick](0,1) circle (1) node[above=0.75]{\rmidarrow} node[above=1]{${\scriptscriptstyle{e_3^1}}$};
\draw[black,thick](0,1.50) circle (1.50) node[above=1.25]{\rmidarrow}node[above=1.50]{${\scriptscriptstyle{e_{n_1}^1}}$};
\draw[loosely dotted, black, thick] (0,2.25)--(0,3);

\draw[fill=black] (4,0) circle (2pt) node[anchor=south]{$v_{2}$};
\draw[black,thick](4,0.5) circle (0.5) node[above=0.25]{\rmidarrow}node[above=0.5]{${\scriptscriptstyle{e_1^2}}$};
\draw[black,thick](4,0.75) circle (0.75) node[above=0.5]{\rmidarrow} node[above=0.75]{${\scriptscriptstyle{e_2^2}}$};
\draw[black,thick](4,1.25) circle (1.25) node[above=1]{\rmidarrow}node[above=1.25]{${\scriptscriptstyle{e_{n_2}^2}}$};
\draw[loosely dotted, black,thick] (4,2)--(4,2.5);

\draw[fill=black] (12,0) circle (2pt) node[anchor=south]{$v_{m}$};
\draw[black,thick](12,0.5) circle (0.5) node[above=0.25]{\rmidarrow}node[above=0.5]{${\scriptscriptstyle{e_1^m}}$};
\draw[black,thick](12,1) circle (1) node[above=0.75]{\rmidarrow}node[above=1]{${\scriptscriptstyle{e_{n_m}^m}}$};
\draw[loosely dotted, black,thick] (12,1.50)--(12,2.25);

\draw[loosely dotted, black,thick] (7,0)--(9,0);
\end{tikzpicture}
\caption{$\sqcup_{i=1}^{m} L_{n_i}$}      \label{L_ni_distinct}
\end{figure} 
\end{center}
\begin{proof}
Let $V(L_{n_i})=\{i\}$ for all $i \in [m]$ and $E_i:=E(L_{n_i})=\{e_1^i, e_2^i,...,e_{n_i}^i\}$ (see Figure \ref{L_ni_distinct}). Therefore, $E:=E(\sqcup_{i=1}^{m} L_{n_i})=\sqcup_{i=1}^{m} E_i .$
Without loss of generality, we may assume that $n_1 > n_2> \cdots >n_m$.\\
 The linear action $\alpha: C^*(\sqcup_{i=1}^{m} L_{n_i}) \to C^*(\sqcup_{i=1}^{m} L_{n_i})  \otimes Q_{\tau}^{Lin}(\sqcup_{i=1}^{m} L_{n_i} )$ is given by  $\alpha(S_{e^j_l})=\sum\limits_{f \in E} S_f \otimes q_{f e^j_l}$.
For each $i \neq j $, it's enough to show that $q_{e^i_k e^j_l}=0 ~~ \forall  k \in [n_i] $ and $\forall l \in [n_j] $  because in this scenario the above action reduces to $\alpha(S_{e^j_l})=\sum\limits_{f \in E_j} S_f \otimes q_{f e^j_l}=\sum\limits_{k \in [n_j]} S_{e^j_k} \otimes q_{e^j_k e^j_l}$ and the result follows using Proposition 4.12 of \cite{Mandal}. Since $q_{e^i_k e^j_l}=0 \Rightarrow \kappa(q_{e^i_k e^j_l})=0 \Rightarrow q_{e^j_l e^i_k}=0 $, it suffices to show $q_{e^i_k e^j_l}=0 ~~ \forall k \in [n_i]  $ and $\forall l \in [n_j]$ with $i<j$. \\
We will show it inductively. Let\\[0.2cm]
$\mathcal{S}(i): q_{e^i_k f}=0 ~~  \forall k \in [n_i] \text{ and } f \in \sqcup_{j=i+1}^{m} E_j $.\\[0.2cm] 
Also, for our convenience, we define  $Q_{ef}= q_{ef}^*q_{ef}$.\\
Now, we will show that $\mathcal{S}(1)$ is true. Since
\begin{align*} 
& \sum_{i=1}^{m} p_{i}=1, \\
\Rightarrow ~~& S_{e_{j_1}^1}^{*} S_{e_{j_1}^1}+S_{e_{j_2}^2}^{*}S_{e_{j_2}^2}+ \cdots + S_{e_{j_m}^m}^{*}S_{e_{j_m}^m} =1 ~~~~\forall j_i \in [n_i] \text{ and for each } i \in [m]  \\
\Rightarrow ~~& \alpha(S_{e_{j_1}^1}^{*} S_{e_{j_1}^1}+S_{e_{j_2}^2}^{*}S_{e_{j_2}^2}+ \cdots + S_{e_{j_m}^m}^{*}S_{e_{j_m}^m} ) =1 \otimes 1 ~~~~\forall j_i \in [n_i] \text{ and for each } i \in [m] \\
\Rightarrow ~~& \sum_{e \in E} S_e^{*}S_e \otimes (Q_{ee_{j_1}^1}+ Q_{ee_{j_2}^2}+ \cdots + Q_{ee_{j_m}^m})=1 \otimes 1 ~~~~\forall j_i \in [n_i] \text{ and for each } i \in [m]. 
\end{align*}
Now, for each $i \in [m]$, multiplying both sides by $(p_i \otimes 1) $, we get $m$ equations of the form
\begin{equation}
\sum_{k \in [n_i]} (Q_{e^{i}_{k} e_{j_1}^1}+ Q_{e^{i}_{k} e_{j_2}^2}+ \cdots + Q_{e^{i}_{k} e_{j_m}^m})=1 . \label{niso1}
\end{equation}
For $i=1$, we get
\begin{equation*}
\sum_{k \in [n_1]} (Q_{e^{1}_{k} e_{j_1}^1}+ Q_{e^{1}_{k} e_{j_2}^2}+ \cdots + Q_{e^{1}_{k} e_{j_m}^m})=1
\end{equation*}
for each $j_t \in [n_t]$.\\
Now taking the sum over $j_t \in [n_t]$ for each $t \in [m]$, we get
\begin{align} \label{niso1.5}
& \sum_{j_m \in [n_m]} \cdots \sum_{j_2 \in [n_2]} \sum_{j_1 \in [n_1]} \left[ \sum_{k \in [n_1]} (Q_{e^{1}_{k} e_{j_1}^1}+ Q_{e^{1}_{k} e_{j_2}^2}+ \cdots + Q_{e^{1}_{k} e_{j_m}^m}) \right]   = \prod_{i=1}^{m} n_i \nonumber \\
\Rightarrow ~~&  \sum_{i=1}^{m}  \left\lbrace \left(  \prod_{k \in [m], k \neq i} n_k \right) \sum_{k \in [n_1]}\sum_{j_i \in [n_i] }Q_{e^{1}_{k} e_{j_i}^i}\right\rbrace  = \prod_{i=1}^{m} n_i \nonumber \\
\Rightarrow ~~& \frac{1}{n_1} \sum_{k \in [n_1]}\sum_{ j_1 \in [n_1] }Q_{e^{1}_{k} e_{j_1}^1}+ \frac{1}{n_2} \sum_{k \in [n_1]}\sum_{ j_2 \in [n_2] }Q_{e^{1}_{k} e_{j_2}^2}+ \cdots + \frac{1}{n_m} \sum_{k \in [n_1]}\sum_{ j_m \in [n_m] }Q_{e^{1}_{k} e_{j_m}^m}=1 
\end{align}

\begin{multline}
\Rightarrow ~~  \frac{1}{n_1} \sum_{k \in [n_1]} \left[ \sum_{ j_1 \in [n_1] }Q_{e^{1}_{k} e_{j_1}^1}+ \sum_{ j_2 \in [n_2] }Q_{e^{1}_{k} e_{j_2}^2} +\cdots + \sum_{ j_m \in [n_m] }Q_{e^{1}_{k} e_{j_m}^m} \right] + \\
\left( \frac{1}{n_2}-\frac{1}{n_1} \right) \sum_{k \in [n_1]}\sum_{ j_2 \in [n_2] }Q_{e^{1}_{k} e_{j_2}^2}+ \cdots + \left( \frac{1}{n_m}-\frac{1}{n_1} \right) \sum_{k \in [n_1]}\sum_{ j_m \in [n_m] }Q_{e^{1}_{k} e_{j_m}^m}=1 \label{niso2}
\end{multline}
Since for each $k \in [n_1]$, $\left[ \sum\limits_{ j_1 \in [n_1] }Q_{e^{1}_{k} e_{j_1}^1}+ \sum\limits_{ j_2 \in [n_2] }Q_{e^{1}_{k} e_{j_2}^2} +\cdots + \sum\limits_{ j_m \in [n_m] }Q_{e^{1}_{k} e_{j_m}^m} \right]=1 $ (as $U^t$ is unitary), therefore the above equation can be written as
\begin{align*}
& \frac{1}{n_1} \sum_{k \in [n_1]} 1 + 
\left( \frac{1}{n_2}-\frac{1}{n_1} \right) \sum_{k \in [n_1]}\sum_{ j_2 \in [n_2] }Q_{e^{1}_{k} e_{j_2}^2}+ \cdots + \left( \frac{1}{n_m}-\frac{1}{n_1} \right) \sum_{k \in [n_1]}\sum_{ j_m \in [n_m] }Q_{e^{1}_{k} e_{j_m}^m}=1  \\
\Rightarrow ~~& \left( \frac{1}{n_2}-\frac{1}{n_1} \right) \sum_{k \in [n_1]}\sum_{ j_2 \in [n_2] }Q_{e^{1}_{k} e_{j_2}^2}+ \cdots + \left( \frac{1}{n_m}-\frac{1}{n_1} \right) \sum_{k \in [n_1]}\sum_{ j_m \in [n_m] }Q_{e^{1}_{k} e_{j_m}^m}=0 ~~ {(\because \sum_{k \in [n_1]} 1 = n_1 )}.
\end{align*}
Since $\left( \frac{1}{n_i}-\frac{1}{n_1} \right) > 0$ for all $i \in \{2,3,...m\}$, therefore we can conclude that $ Q_{e^{1}_{k} e_{j_2}^2}=Q_{e^{1}_{k} e_{j_3}^3}=...=Q_{e^{1}_{k} e_{j_m}^m}=0$ for all $k \in [n_1]$ and $ j_i \in [n_i]$ whenever $i \in \{2,...,m\}$.\\
In other words, $q_{e^{1}_{k} f}=0$ for all $k \in [n_1]$ and $f \in \sqcup_{i=2}^{m}E_i $. Therefore, $\mathcal{S}(1)$ is true.\\

By induction hypothesis, we assume that $\mathcal{S}(i)$ is true for $i=1,2,...,l-1 $, i.e. for each  $i=1,2,...,l-1, ~ q_{e^{i}_{k} f}=0 ~~ \forall k \in [n_i]$ and $f \in \sqcup_{s   =i+1}^{m}E_s $.\\
Now, we have to show that $\mathcal{S}(l)$ is also true.\\

\noindent Now, putting $i=l$ in Equation $\eqref{niso1}$, we get
\begin{equation*}
\sum_{k \in [n_l]} (Q_{e^{l}_{k} e_{j_1}^1}+ Q_{e^{l}_{k} e_{j_2}^2}+ \cdots + Q_{e^{l}_{k} e_{j_m}^m})=1
\end{equation*} 
Now, by induction hypothesis, for each $k \in [n_l]$, $Q_{e^{l}_{k} e_{j_1}^1}=Q_{e^{l}_{k} e_{j_2}^2}=...=Q_{e^{l}_{k} e_{j_{l-1}}^{l-1}}=0 $ for any ${j_1} \in [n_1],...,{j_{l-1}} \in [n_{l-1}] .$ Therefore the above equation reduces into 
\begin{equation*}
\sum_{k \in [n_l]} (Q_{e^{l}_{k} e_{j_l}^l}+ Q_{e^{l}_{k} e_{j_{l+1}}^{l+1}}+ \cdots + Q_{e^{l}_{k} e_{j_m}^m})=1
\end{equation*} 
Now taking the sum over $j_i \in [n_i]$ for each $i \in \{l,(l+1),...,m\}$, we get
\begin{align}
& \sum_{j_m \in [n_m]} \cdots \sum_{j_{l+1} \in [n_{l+1}]} \sum_{j_l \in [n_l]} \left[ \sum_{k \in [n_l]} (Q_{e^{l}_{k} e_{j_l}^l}+ Q_{e^{l}_{k} e_{j_{l+1}}^{l+1}}+ \cdots + Q_{e^{l}_{k} e_{j_m}^m})\right]   = \prod_{i=l}^{m} n_i \nonumber \\
\Rightarrow ~~& \frac{1}{n_l} \sum_{k \in [n_l] }\sum_{ j_l \in [n_l] }Q_{e^{l}_{k} e_{j_l}^1}+ \frac{1}{n_{l+1}} \sum_{k \in [n_l]}\sum_{ j_{l+1} \in [n_{l+1}] }Q_{e^{l}_{k} e_{j_{l+1}}^{l+1}}+ \cdots + \frac{1}{n_m} \sum_{k \in [n_l]}\sum_{ j_m \in [n_m] }Q_{e^{l}_{k} e_{j_m}^m}=1. \label{niso2.5}
\end{align}
\begin{multline}
\therefore  \frac{1}{n_l} \sum_{k \in [n_l]} \left[ \sum_{ j_l \in [n_l] }Q_{e^l_k e_{j_l}^l}+ \sum_{ j_{l+1} \in [n_{l+1}] }Q_{e^l_k e_{j_{l+1}}^{l+1}} +\cdots + \sum_{ j_m \in [n_m] }Q_{e^l_k e_{j_m}^m} \right] + \\
\left( \frac{1}{n_{l+1}}-\frac{1}{n_l} \right) \sum_{k \in [n_l]}\sum_{ j_{l+1} \in [n_{l+1}] }Q_{e^l_k e_{j_{l+1}}^{l+1}}+ \cdots + \left( \frac{1}{n_m}-\frac{1}{n_l} \right) \sum_{k \in [n_l]}\sum_{ j_m \in [n_m] }Q_{e^l_k e_{j_m}^m}=1.  \label{niso3}
\end{multline}
Using the fact that $U^t=(q_{ef})^t$ is unitary and the induction hypothesis again, one can easily find that
     $$\left[ \sum_{ j_l \in [n_l] }Q_{e^l_k e_{j_l}^l}+ \sum_{ j_{l+1} \in [n_{l+1}] }Q_{e^l_k e_{j_{l+1}}^{l+1}} +\cdots + \sum_{ j_m \in [n_m] }Q_{e^l_k e_{j_m}^m} \right]=1 .$$ 
Therefore, one can rewrite Equation \eqref{niso3} as before:
\begin{equation*}
   \left( \frac{1}{n_{l+1}}-\frac{1}{n_l} \right) \sum_{k \in [n_l]}\sum_{ j_{l+1} \in [n_{l+1}] }Q_{e^l_k e_{j_{l+1}}^{l+1}}+ \cdots + \left( \frac{1}{n_m}-\frac{1}{n_l} \right) \sum_{k \in [n_l]}\sum_{ j_m \in [n_m] }Q_{e^l_k e_{j_m}^m}=0
\end{equation*}  
Hence, for all ${k \in [n_l]}$, we get  $ Q_{e^l_k e_{j_{l+1}}^{l+1}}=...=Q_{e^l_k e_{j_m}^m}=0$ for all ${j_{i}} \in [n_i] $ with $i \in \{l+1,...,m\}.$ 
Therefore, $\mathcal{S}(l)$ is also true.  
\end{proof}	 
  
  \begin{rem}
  	Though the statement of Theorem \ref{thm1} is written in the sense of `CGQ isomorphic', the proof suggests that the result is also true in `identical' sense, i.e., under the hypothesis of Theorem \ref{thm1}, we have $ Q_{\tau}^{Lin}(\sqcup_{i=1}^{m} L_{n_i} ) \approx *_{i=1}^{m} Q_{\tau}^{Lin}(L_{n_i}) \approx *_{i=1}^{m} U_{n_i}^{+} $.
  \end{rem}
  
  \begin{rem}
  	It is known that if we take $m$ connected graphs $\{\Gamma_{i}\}_{i=1}^{m}$ which are `quantum non-isomorphic' to each other, then 
  	\begin{equation*}
  	QAut_{Ban}(\sqcup_{i=1}^{m} ~ \Gamma_{i}) \approx *_{i=1}^{m} ~~ QAut_{Ban}(\Gamma_{i})
  	\end{equation*} 
  	with respect to their standard fundamental matrix representation (consult \cite{Lupini}, \cite{SSthesis}, \cite{QAut_tree} and \cite{Meu} for details).  From Theorem \ref{thm1},  for non-isomorphic Cuntz algebras $\{ \mathcal{O}_{n_i}\}$ we get that
  	\begin{equation*}
  	Q_{\tau}^{Lin}(\sqcup_{i=1}^{m} ~ L_{n_i}) \approx  *_{i=1}^{m} ~~ Q_{\tau}^{Lin}(L_{n_i}).
  	\end{equation*} 
\noindent It is natural to ask: does a similar result hold if we replace $\{L_{n_i}\}$ by any non-isomorphic class of graphs? But the answer is negative in graph $C^*$-algebraic context, even if we take `quantum non-isomorphic' graphs (see the counter example in Subsection \ref{counternonisom}). But we encountered that for non-isomorphic matrix algebras $\{M_{n_i}\}$ (for distinct $n_i$'s), we also have an analogous result, i.e. 
  	\begin{equation*}
  	Q_{\tau}^{Lin}(\sqcup_{i=1}^{m} ~ P_{n_i}) \approx  *_{i=1}^{m} ~~ Q_{\tau}^{Lin}(P_{n_i}). 
  	\end{equation*}
  The proof is a simple application of Lemma 4.4 and Lemma 4.5 of \cite{rigidity}. In this article, we omit the details of the proof.	 	
  \end{rem}

\begin{subsection}{Counter Example} \label{counternonisom}
In this subsection, we provide two non-isomorphic graphs $\Gamma_1$ and $\Gamma_2$ to show that $Q_{\tau}^{Lin}(\Gamma_1 \sqcup \Gamma_2)$ is not always identical to   $Q_{\tau}^{Lin}(\Gamma_1) * Q_{\tau}^{Lin}(\Gamma_2)$ in general.\\

\noindent Consider the graph $P_1 \sqcup So_2$ (as it is shown in Figure \ref{P1So2}), which is the disjoint union of two non-isomorphic (even `quantum non-isomorphic') graphs $P_1$ and $So_2$, where $P_1$ and $So_2$ denote the first and second connected components of the graph shown in Figure \ref{P1So2}.\\

\begin{figure}[htpb]
	\centering
	\begin{tikzpicture}
		\draw[fill=black] (0,0) circle (2pt) node[anchor=north]{$v_{1}$};
		\draw[fill=black] (2,0) circle (2pt) node[anchor=north]{$v_{2}$};
		\draw[fill=black] (4,0) circle (2pt) node[anchor=north]{$v_{3}$};
		\draw[fill=black] (6,0) circle (2pt) node[anchor=north]{$v_{4}$};
		\draw[fill=black] (8,0) circle (2pt) node[anchor=north]{$v_{5}$};
		
		\draw[black] (0,0)-- node{\rmidarrow}  node[above]{$e_{1}$} (2,0);
		\draw[black] (4,0)-- node{\lmidarrow}  node[above]{$e_{2}$} (6,0);
		\draw[black] (6,0)-- node{\rmidarrow}  node[above]{$e_{3}$} (8,0);
	\end{tikzpicture}
	\caption{$P_1 \sqcup So_2$} \label{P1So2}
\end{figure}
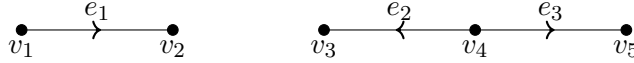

\noindent By Proposition 3.4 of \cite{unitaryeasygraph}, $(Q_{\tau}^{Lin}(P_1 \sqcup So_2),q) \approx SH_{3}^{\infty +}$ and $(Q_{\tau}^{Lin}(So_2),q) \approx SH_{2}^{\infty +}$.
Thus, it is evident that $Q_{\tau}^{Lin}(P_1 \sqcup So_2) \approx SH_{3}^{\infty +}$ is not identical to $C(S^1)*SH_{2}^{\infty +} \approx Q_{\tau}^{Lin}(P_1)*Q_{\tau}^{Lin}(So_2)$, where all CMQGs are taken with respect to their standard fundamental representations as it is introduced in Subsections \ref{CQG} and \ref{QSgraph}.\\ 
Now, we remark that in the context of the quantum automorphism group of a graph (in the sense of Banica)  $QAut_{Ban}(P_1 \sqcup So_2) \approx S_2^+ \approx QAut_{Ban}(P_1) * QAut_{Ban}(So_2) $, whereas in the graph $C^*$-algebraic scenario $Q_{\tau}^{Lin}(P_1 \sqcup So_2)$ is not identical to $Q_{\tau}^{Lin}(P_1) * Q_{\tau}^{Lin}(So_2)$.\\
\end{subsection}

Now, we demonstrate that an analogous result for Theorem \ref{thm1} also holds with respect to the `KMS state' and the `orthogonal filtration' preserving categories for the direct sum of Cuntz algebras.
  
\begin{thm}
Let $\{\mathcal{O}_{n_i}\}_{i=1}^{m}$ be a finite family of Cuntz algebras, where all $n_i$'s are distinct. Then $ Q_{\oplus KMS}^{Lin}(\sqcup_{i=1}^{m} L_{n_i} )  \cong  *_{i=1}^m U_{n_i}^{+} \cong *_{i=1}^{m} Q_{KMS}^{Lin}(L_{n_i}). $  \label{diskmsthm}
\end{thm}
\begin{proof}
We will use the same notations and conventions introduced in the proof of Theorem \ref{thm1}. Firstly, we will show that if $(Q_{\oplus KMS}^{Lin}(\sqcup_{i=1}^{m} L_{n_i}), \alpha)$ is the universal object with the fundamental representation $U=(q_{e_k^i e_l^j})_{n \times n}$ of underlying CMQG, then $q_{e_k^i e_l^j}=0$ for all $k \in [n_i]$ and $l \in [n_j]$, whenever $i \neq j$. The strategy of the proof is again similar to Theorem \ref{thm1}, the method of induction (see the induction statement $\mathcal{S}(i)$ from Theorem \ref{thm1}). Moreover, the proof follows the same lines, but we just need to modify a few computational tricks, which we will discuss here. To show $\mathcal{S}(1)$, just like the previous theorem, again starting with  $\sum\limits_{i=1}^{m} p_i = 1$, we arrive at Equation \eqref{niso1.5}, i.e.
\begin{align*}
	& \frac{1}{n_1} \sum_{k \in [n_1]}\sum_{ j_1 \in [n_1] }Q_{e^{1}_{k} e_{j_1}^1}+ \frac{1}{n_2} \sum_{k \in [n_1]}\sum_{ j_2 \in [n_2] }Q_{e^{1}_{k} e_{j_2}^2}+ \cdots + \frac{1}{n_m} \sum_{k \in [n_1]}\sum_{ j_m \in [n_m] }Q_{e^{1}_{k} e_{j_m}^m}=1, \\
	\Rightarrow ~~ & \frac{1}{n_1^2} \sum_{k \in [n_1]}\sum_{ j_1 \in [n_1] } n_1 Q_{e^{1}_{k} e_{j_1}^1}+ \frac{1}{n_2^2} \sum_{k \in [n_1]}\sum_{ j_2 \in [n_2] } n_2 Q_{e^{1}_{k} e_{j_2}^2}+ \cdots + \frac{1}{n_m^2} \sum_{k \in [n_1]}\sum_{ j_m \in [n_m] } n_m Q_{e^{1}_{k} e_{j_m}^m}=1, 
\end{align*} 
which implies 
\begin{multline}
	\frac{1}{n_1^2} \sum_{k \in [n_1]} \left[ \sum_{ j_1 \in [n_1] }n_1 Q_{e^{1}_{k} e_{j_1}^1}+ \sum_{ j_2 \in [n_2] } n_2 Q_{e^{1}_{k} e_{j_2}^2} +\cdots + \sum_{ j_m \in [n_m] } n_m Q_{e^{1}_{k} e_{j_m}^m} \right] + \\
	\left( \frac{1}{n_2^2}-\frac{1}{n_1^2} \right) \sum_{k \in [n_1]}\sum_{ j_2 \in [n_2] } n_2 Q_{e^{1}_{k} e_{j_2}^2}+ \cdots + \left( \frac{1}{n_m^2}-\frac{1}{n_1^2} \right) \sum_{k \in [n_1]}\sum_{ j_m \in [n_m] } n_m Q_{e^{1}_{k} e_{j_m}^m}=1. \label{nisokms1}
\end{multline}
Since  $\left[ \sum\limits_{ j_1 \in [n_1] } n_1 Q_{e^{1}_{k} e_{j_1}^1}+ \sum\limits_{ j_2 \in [n_2] } n_2 Q_{e^{1}_{k} e_{j_2}^2} +\cdots + \sum\limits_{ j_m \in [n_m] } n_m Q_{e^{1}_{k} e_{j_m}^m} \right]=n_1 $  for each $k \in [n_1]$ (using Equation \eqref{relUtF}), Equation \eqref{nisokms1} reduces to
\begin{align*}
	& \frac{1}{n_1^2} \sum_{k \in [n_1]} n_1 + 
	\left( \frac{1}{n_2^2}-\frac{1}{n_1^2} \right) \sum_{k \in [n_1]}\sum_{ j_2 \in [n_2] } n_2 Q_{e^{1}_{k} e_{j_2}^2}+ \cdots + \left( \frac{1}{n_m^2}-\frac{1}{n_1^2} \right) \sum_{k \in [n_1]}\sum_{ j_m \in [n_m] } n_m Q_{e^{1}_{k} e_{j_m}^m}=1  \\
	\Rightarrow ~~& \left( \frac{1}{n_2^2}-\frac{1}{n_1^2} \right) \sum_{k \in [n_1]}\sum_{ j_2 \in [n_2] } n_2 Q_{e^{1}_{k} e_{j_2}^2}+ \cdots + \left( \frac{1}{n_m^2}-\frac{1}{n_1^2} \right) \sum_{k \in [n_1]}\sum_{ j_m \in [n_m] } n_m Q_{e^{1}_{k} e_{j_m}^m}=0.
\end{align*}
Now, $\left( \frac{1}{n_i^2}-\frac{1}{n_1^2} \right) > 0$ for all $i \in \{2,3,...m\}$ ensures that $ Q_{e^{1}_{k} e_{j_2}^2}=Q_{e^{1}_{k} e_{j_3}^3}=...=Q_{e^{1}_{k} e_{j_m}^m}=0$ for all $k \in [n_1]$ and $ j_i \in [n_i]$ whenever $i \in \{2,...,m\}$, which concludes the base case of induction $\mathcal{S}(1)$. \\
Assuming $\mathcal{S}(i)$ holds for each $i \in [l-1]$, we are now going to show that $\mathcal{S}(l)$ is true. Just like the proof of Theorem \ref{thm1}, starting with Equation \eqref{niso1} (by putting $i=l$) and using the induction hypothesis, we get Equation \eqref{niso2.5}. Now, Equation \eqref{niso2.5} implies     
\begin{multline}
	\therefore  \frac{1}{n_l^2} \sum_{k \in [n_l]} \left[ \sum_{ j_l \in [n_l] } n_l Q_{e^l_k e_{j_l}^l}+ \sum_{ j_{l+1} \in [n_{l+1}] } n_{l+1} Q_{e^l_k e_{j_{l+1}}^{l+1}} +\cdots + \sum_{ j_m \in [n_m] } n_m Q_{e^l_k e_{j_m}^m} \right] + \\
	\left(\frac{1}{n_{l+1}^2}-\frac{1}{n_l^2} \right) \sum_{k \in [n_l]}\sum_{ j_{l+1} \in [n_{l+1}] } n_{l+1} Q_{e^l_k e_{j_{l+1}}^{l+1}}+ \cdots + \left( \frac{1}{n_m^2}-\frac{1}{n_l^2} \right) \sum_{k \in [n_l]}\sum_{j_m \in [n_m] } n_m Q_{e^l_k e_{j_m}^m}=1  \label{nisokms2}
\end{multline}
(which is just a modification of Equation \eqref{niso3}).\\
Again using the induction hypothesis, Equation \eqref{relUtF} reduces to 
 $$\left[ \sum_{ j_l \in [n_l] } n_l Q_{e^l_k e_{j_l}^l}+ \sum_{ j_{l+1} \in [n_{l+1}]} n_{l+1} Q_{e^l_k e_{j_{l+1}}^{l+1}} +\cdots + \sum_{ j_m \in [n_m] } n_m Q_{e^l_k e_{j_m}^m} \right]=n_l .$$ 
Therefore, Equation \eqref{nisokms2} implies
\begin{equation*}
	\left( \frac{1}{n_{l+1}^2}-\frac{1}{n_l^2} \right) \sum_{k \in [n_l]}\sum_{ j_{l+1} \in [n_{l+1}] } n_{l+1} Q_{e^l_k e_{j_{l+1}}^{l+1}}+ \cdots + \left( \frac{1}{n_m^2}-\frac{1}{n_l^2} \right) \sum_{k \in [n_l]}\sum_{ j_m \in [n_m] } n_m Q_{e^l_k e_{j_m}^m}=0
\end{equation*}  
Hence, for each ${k \in [n_l]}$, we have  $ Q_{e^l_k e_{j_{l+1}}^{l+1}}=...=Q_{e^l_k e_{j_m}^m}=0$ for all ${j_{i}} \in [n_i] $ with $i \in \{l+1,...,m\}.$ 
Therefore, $\mathcal{S}(l)$ is also true.\\  
Now, using the universal property of $*_{i=1}^m U_{n_i}^{+}$, one can conclude that $Q_{\oplus KMS}^{Lin}(\sqcup_{i=1}^{m} L_{n_i})$ is a quantum subgroup of $*_{i=1}^m U_{n_i}^{+}$ with respect to their standard fundamental representation. \\

\noindent  Conversely, we will show that the CMQG $*_{i=1}^{m} U_{n_i}^+$ (whose fundamental representation is given by $q=(q_{ef})_{n \times n}$, where $q$ is a unitary matrix and $q_{ef}=0$ iff $e$ and $f$ based on two distinct vertices) acts linearly, faithfully and $\oplus_{i=1}^{m} KMS_{ln(n_i)}$ preserving way on $C^*(\sqcup_{i=1}^{m} ~~ L_{n_i})$, where we denote $ \sqcup_{i=1}^{m} ~~ L_{n_i}=\{V, E, s, r\}$ with $|V|=m$ and $n:=|E|=\sum\limits_{i=1}^{m} n_i$. Also, we write $E_k:=\{e \in E : s(e)=r(e)=k \}$. Let us define a linear faithful action $\alpha$ by 
\begin{equation}
\alpha(S_e)= \sum\limits_{f \in E_i} S_f \otimes q_{ef} \text{ for all } e \in E_i \text{ and for each } i \in [m].  \label{actioneq}
\end{equation}
 We have to show that $\alpha$ preserves the state $\oplus_{i=1}^{m} KMS_{ln(n_i)} $.\\
 Let  $\gamma, \mu$ be finite paths in $ \sqcup_{i=1}^{m} ~~ L_{n_i}$. Hence, $\gamma, \mu$ are paths in $L_{n_k}$ and $L_{n_l}$ respectively, for some $k,l\in [m]$.\\
    If $k\neq l$, then $S_{\gamma}S_{\mu}^*=0$. Therefore,  $( \oplus_{i=1}^{m} KMS_{ln(n_i)} \otimes id_{C^*(\Gamma)} ) \circ \alpha (S_{\gamma}S_{\mu}^*) =0=  \oplus_{i=1}^{m} KMS_{ln(n_i)}(S_{\gamma}S_{\mu}^*).1 $ follows trivially.\\
	If $k=l$, then $S_{\gamma}S_{\mu}^* \neq 0$ (since $r(\gamma)=r(\mu)$). Let us assume $\gamma= \gamma_1 \gamma_2 ...\gamma_p$ and $\mu= \mu_1 \mu_2 ...\mu_r$.
	\begin{equation*}
	\alpha(S_{\gamma}S_{\mu}^*)=\sum_{i_1,...,i_p;j_1,..., j_r \in E_k} S_{i_1}...S_{i_p}S_{j_r}^*...S_{j_1}^* \otimes q_{i_1 \gamma_1}...q_{i_p \gamma_p} q_{j_r \mu_r}^*...q_{j_1 \mu_1}^*.
	\end{equation*}
	\underline{Case 1:} $p \neq r$.  $\oplus_{i=1}^{m} KMS_{ln(n_i)}(S_{\gamma}S_{\mu}^*)= \frac{1}{m} KMS_{ln(n_k)}(S_{\gamma}S_{\mu}^*)=0  $.\\
	On the other hand, $S_{i_1}...S_{i_p}S_{j_r}^*...S_{j_1}^* \neq 0$, since all the loops $i_1,...i_p, j_1,...j_r $ are based at a single vertex $k$. By definition,
	\begin{equation*}
	KMS_{ln(n_k)}(S_{i_1}...S_{i_p}S_{j_r}^*...S_{j_1}^*)=0.
	\end{equation*}
	Hence, 
	\begin{equation*}
	(\oplus_{i=1}^{m} KMS_{ln(n_i)} \otimes id_{C^*(\Gamma)} ) \circ \alpha (S_{\gamma}S_{\mu}^*) = 0= \oplus_{i=1}^{m} KMS_{ln(n_i)}(S_{\gamma}S_{\mu}^*).1. 
	\end{equation*}
	\underline{Case 2:}  $p=r$. In this case, 
	\begin{align}
	( \oplus_{i=1}^{m} KMS_{ln(n_i)} \otimes id_{C^*(\Gamma)} ) \circ \alpha (S_{\gamma}S_{\mu}^*) =&  \sum_{i_1,...,i_p\in E_k} \frac{1}{m.n_{k}^p} ~ q_{i_1 \gamma_1}...q_{i_p \gamma_p} q_{i_p \mu_p}^*...q_{i_1 \mu_1}^*   \nonumber \\
	=& \frac{1}{m.{n_k}^p}  \sum_{i_1,...,i_p\in E_k}  ~ q_{i_1 \gamma_1}...q_{i_p \gamma_p} q_{i_p \mu_p}^*...q_{i_1 \mu_1}^* \label{KMSsum}
	\end{align}
 Depending on the paths $\gamma, \mu$, we divide it into the following two sub-cases :\\
 
\noindent Case(A): $\gamma = \mu$. Then $\gamma_t= \mu_t$ for all $t \in [p]$.\\
	Since $\sum\limits_{i_t\in E_k}  ~ q_{i_t \gamma_t} q_{i_t \gamma_t}^*=1$ for each $t \in [p]$, $$ \sum_{i_1,...,i_p\in E_k}  ~ q_{i_1 \gamma_1}...q_{i_p \gamma_p} q_{i_p \gamma_p}^*...q_{i_1 \gamma_1}^*=1 .$$ 
	Therefore, from Equation \eqref{KMSsum},
	$ ( \oplus_{i=1}^{m} KMS_{ln(n_i)} \otimes id_{C^*(\Gamma)} ) \circ \alpha (S_{\gamma}S_{\gamma}^*) = \frac{1}{m.{n_k}^p}=\oplus_{i=1}^{m} KMS_{ln(n_i)}(S_{\gamma}S_{\gamma}^*).1$.\\
	
	\noindent Case(B): $\gamma \neq \mu$. So, let's assume that $\gamma_t=\mu_t$ for $t=p,p-1,...p-j$ but $\gamma_{p-j-1}=\mu_{p-j-1}$.
	Now, since $ \sum\limits_{i_t\in E_k}  ~ q_{i_t \gamma_t} q_{i_t \mu_t}^*=1$ for $t=p,p-1,...,p-j$ and $\sum\limits_{i_{p-j-1} \in E_k}  ~ q_{i_{p-j-1} \gamma_{p-j-1}} q_{i_{p-j-1} \mu_{p-j-1}}^*=0$, the sum on the RHS of Equation \eqref{KMSsum} is 0. Therefore,
	\begin{equation*}
	( \oplus_{i=1}^{m} KMS_{ln(n_i)} \otimes id_{C^*(\Gamma)} ) \circ \alpha (S_{\gamma}S_{\mu}^*) = 0 = \frac{1}{m} KMS_{ln(n_k)}(S_{\gamma}S_{\mu}^*).1=\oplus_{i=1}^{m} KMS_{ln(n_i)}(S_{\gamma}S_{\mu}^*).1.
	\end{equation*}
\end{proof}

\noindent Recall that $$\mathfrak{F}=(\{ \mathcal{W}_{i}, \mathcal{M}_{p,q}^{(1)}, \mathcal{M}_{r,s}^{(2)} : i \in \mathbb{N}_{0}, (p,q) \in \mathbb{N}_{0} \times \mathbb{N}, (r,s) \in \mathbb{N}_{0} \times \mathbb{N} \}, \oplus_{k=1}^{m} KMS_{ln(n_k)})$$ is an orthogonal filtration on $\oplus_{i=1}^{m} \mathcal{O}_{n_i}$ and $\mathfrak{F}_{i}$ denotes the orthogonal filtration on $\mathcal{O}_i \cong C^*(L_{n_i})$ with respect to the KMS state $KMS_{ln(n_i)}$ (for $i \in [m]$) as it was introduced in Section 5 of \cite{Mandalkms}.  \\
The same idea, as it was used in the proof of Theorem 5.4 in \cite{Mandalkms}, has been used to prove the following theorem (Theorem \ref{orthothm1}).\\

\noindent 	Let's denote $\sqcup_{i=1}^{m}~L_{n_i}=\{V, E, s, r\}$, where $|V|=m$. Before going to the proof, observe that since  $\mathcal{M}_{0,1}^{(1)}$ generates $C^*(\sqcup_{i=1}^{m}~L_{n_i})$ (as a $C^*$-algebra) and $\{\sqrt{m}S_e : e \in E\}$ is an orthonormal basis for $\mathcal{M}_{0,1}^{(1)}$, Theorem 2.10 [$(ii)$ and $(iii)$] of \cite{Ortho} ensures that $Q_{\mathfrak{F}}({C^*(\sqcup_{i=1}^{m}L_{n_i}))}$ is essentially a CMQG with the fundamental representation $q=(q_{ef})_{|E|\times |E|}$ such that $\alpha(S_e)=\sum\limits_{f \in E} S_f \otimes q_{fe}$.\\

\begin{thm} \label{orthothm1}
	Let $\{\mathcal{O}_{n_i}\}_{i=1}^{m}$ be a finite family of Cuntz algebras, where all $n_i$'s are distinct. Then $Q_{\mathfrak{F}}({C^*(\sqcup_{i=1}^{m}L_{n_i}))} \cong *_{i=1}^{m} U_{n_i}^+ \cong *_{i=1}^{m} Q_{\mathfrak{F}_i}({C^*(L_{n_i}))}$.
\end{thm}
\begin{proof}
	Firstly, we claim that there exists a surjective $C^*$-homomorphism from $*_{i=1}^{m} U_{n_i}^+$ to $Q_{\mathfrak{F}}({C^*(\sqcup_{i=1}^{m}L_{n_i}))}$ that takes generators of $*_{i=1}^{m} U_{n_i}^+$ to generators of $Q_{\mathfrak{F}}({C^*(\sqcup_{i=1}^{m}L_{n_i}))}$. If  $((Q,\Delta),\alpha)$ is an object in the category $\mathfrak{C}_{\mathfrak{F}}({C^*(\sqcup_{i=1}^{m}L_{n_i}))}$, then by Equation (2.1) of \cite{Ortho}, $\alpha$ preserves the underlying state $\oplus_{k=1}^{m} KMS_{ln(n_k)}$. Moreover, $\alpha(\mathcal{M}_{0,1}^{(1)}) \subset \mathcal{M}_{0,1}^{(1)} \otimes Q $ implies that $\alpha$ is linear. Hence, $((Q,\Delta),\alpha)$ is an object of $\mathfrak{C}_{\oplus KMS}^{Lin}$ and the claim follows from Theorem \ref{diskmsthm}. \\
	
	\noindent For the converse, we have to show that the CMQG $(*_{i=1}^{m} U_{n_i}^+,q)$ (as it was described in Theorem \ref{diskmsthm}) acts on $C^*(\sqcup_{i=1}^{m}L_{n_i})$ by  $\alpha$ as in Equation \eqref{actioneq}, which preserves the orthogonal filtration $\mathfrak{F}=(\{ \mathcal{W}_{i}, \mathcal{M}_{p,q}^{(1)}, \mathcal{M}_{r,s}^{(2)} : i \in \mathbb{N}_{0}, (p,q) \in \mathbb{N}_{0} \times \mathbb{N}, (r,s) \in \mathbb{N}_{0} \times \mathbb{N} \}, \oplus_{k=1}^{m} KMS_{ln(n_k)})$. Using \eqref{actioneq}, we get 
	\begin{equation*}
	\alpha(S_{\gamma}S_{\mu}^*)=\sum_{i_1,...,i_k;j_1,..., j_k \in E_l} S_{i_1}...S_{i_k}S_{j_k}^*...S_{j_1}^* \otimes q_{i_1 \gamma_1}...q_{i_k \gamma_k} q_{j_k \mu_k}^*...q_{j_1 \mu_1}^*,
	\end{equation*}
	where $\gamma= \gamma_1 \gamma_2 ...\gamma_k$ and $\mu= \mu_1 \mu_2 ...\mu_k$ are paths in $L_{n_l}$. Therefore, $\alpha(\mathcal{F}_k) \subset \mathcal{F}_k \otimes Q $. Since the fundamental representation $q=(q_{ef})_{n \times n}$ is unitary, the action restricted on each $\mathcal{F}_{k}$, $\alpha|_{\mathcal{F}_{k}}$, is actually a unitary (co)representation on $\mathcal{F}_{k}$. Hence, $\alpha$ preserves $\mathcal{W}_k= \mathcal{F}_k \ominus \mathcal{F}_{k-1}$, the orthogonal complement of $\mathcal{F}_{k-1}$ in $\mathcal{F}_{k}$. \\
	Lastly, to show $\alpha(\mathcal{M}_{r,s}^{(1)}) \subset \mathcal{M}_{r,s}^{(1)} \otimes Q$, take a non-zero element $S_{\mu}x=S_{\mu_1}...S_{\mu_s}x \in \mathcal{M}_{r,s}^{(1)}$, where $x \in \mathcal{W}_r$ and $\mu$ is a path on $E(L_l)$. Since $\alpha(S_{\mu}x)=\left( \sum\limits_{i_1,...,i_s \in E_l} S_{i_1}...S_{i_s} \otimes q_{i_1 \mu_1}...q_{i_s \mu_s} \right) \left( \alpha(x) \right) $, where $\alpha(x) \in \mathcal{W}_r \otimes Q$, the claim follows. Also, the *-preserving property of $\alpha$ ensures that similar arguments hold for  $\alpha(\mathcal{M}_{r,s}^{(2)}) \subset \mathcal{M}_{r,s}^{(2)} \otimes Q$. Hence, there exists a surjective $C^*$-homomorphism from $Q_{\mathfrak{F}}({C^*(\sqcup_{i=1}^{m}L_{n_i}))}$ to $*_{i=1}^{m} U_{n_i}^+$ that takes generators to generators, which concludes the first isomorphism in the statement of the theorem.\\  
	The second isomorphism follows from Corollary 5.5 of \cite{Mandalkms}.
\end{proof}

\begin{obs} \label{orthothm1rem}
	We would just like to mention two key observations from the above proof that will help us to understand Theorem \ref{orthothm2} in the future. 
	\begin{enumerate}
		\item[(i)] Any object of the category $\mathfrak{C}_{\mathfrak{F}}(C^*(\sqcup_{i=1}^{m} L_{n_i}))$ is also an object of the category $\mathfrak{C}_{\oplus KMS}^{Lin}$ (regardless of whether $n_i$'s are distinct or not).
		\item[(ii)] In the second half of the above proof, to show the defined action $\alpha$ is filtration preserving, we have just used the fact that the fundamental matrix representation $q$ of the acting CMQG is actually a unitary matrix.   
	\end{enumerate}
\end{obs}
\section{Direct sum of isomorphic Cuntz algebras} 

\begin{thm} \label{thm2}
Let $\{O_{N}\}_{i=1}^{K}$ be a finite family of isomorphic Cuntz algebras (with $N$ generators) whose underlying graphs are $L_{N}$. Then $ Q_{\tau}^{Lin}(\sqcup_{i=1}^{K} L_{N} ) \cong Q_{\tau}^{Lin}(L_{N}) \wr_{*} S_{K}^+ \cong U_{N}^{+} \wr_{*} S_{K}^+ . $
\end{thm} 

\begin{center}
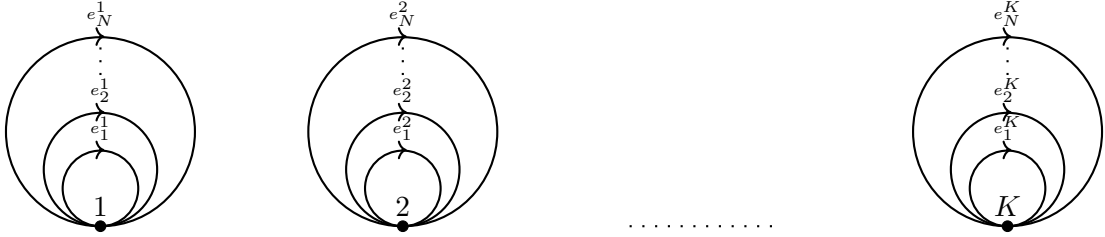
\begin{figure}[htpb]
\begin{tikzpicture}
\draw[fill=black] (0,0) circle (2pt) node[anchor=south]{${1}$};
\draw[black,thick](0,0.5) circle (0.5) node[above=0.25]{\rmidarrow}node[above=0.5]{${\scriptscriptstyle{e_1^1}}$};,
\draw[black,thick](0,0.75) circle (0.75) node[above=0.5]{\rmidarrow} node[above=0.75]{${\scriptscriptstyle{e_2^1}}$};
\draw[black,thick](0,1.25) circle (1.25) node[above=1]{\rmidarrow}node[above=1.25]{${\scriptscriptstyle{e_N^1}}$};
\draw[loosely dotted, black,thick] (0,2)--(0,2.5);

\draw[fill=black] (4,0) circle (2pt) node[anchor=south]{${2}$};
\draw[black,thick](4,0.5) circle (0.5) node[above=0.25]{\rmidarrow}node[above=0.5]{${\scriptscriptstyle{e_1^2}}$};
\draw[black,thick](4,0.75) circle (0.75) node[above=0.5]{\rmidarrow} node[above=0.75]{${\scriptscriptstyle{e_2^2}}$};
\draw[black,thick](4,1.25) circle (1.25) node[above=1]{\rmidarrow}node[above=1.25]{${\scriptscriptstyle{e_N^2}}$};
\draw[loosely dotted, black,thick] (4,2)--(4,2.5);

\draw[fill=black] (12,0) circle (2pt) node[anchor=south]{${K}$};
\draw[black,thick](12,0.5) circle (0.5) node[above=0.25]{\rmidarrow}node[above=0.5]{${\scriptscriptstyle{e_1^K}}$};
\draw[black,thick](12,0.75) circle (0.75) node[above=0.5]{\rmidarrow} node[above=0.75]{${\scriptscriptstyle{e_2^K}}$};
\draw[black,thick](12,1.25) circle (1.25) node[above=1]{\rmidarrow}node[above=1.25]{${\scriptscriptstyle{e_N^K}}$};
\draw[loosely dotted, black,thick] (12,2)--(12,2.5);

\draw[loosely dotted, black,thick] (7,0)--(9,0);
\end{tikzpicture}
\caption{$\sqcup_{i=1}^{K} L_{N}$}    \label{KL_N}
\end{figure}
\end{center}

\begin{proof}
Let us denote $V':=V(\sqcup_{i=1}^{K} L_{N})$ and $E':=E(\sqcup_{i=1}^{K} L_{N})$ (see Figure \ref{KL_N}).
Let the $i^{th}$ loop  based at $a \in V(\sqcup_{i=1}^{K} O_{N})=:V' $ be denoted as $ e^{a}_{i}$ (for $ i \in [N] $).
The action $\alpha$ can be given by 
$$ \alpha(S_{e^b_j})= \sum_{a \in [K], i \in [N]} S_{e^a_i} \otimes q_{e^a_i e^b_j} $$ for $ b \in [K] $ and $ j \in [N]$.  \\
For our convenience, we denote $q_{e^a_i e^b_j}$ by $ q_{ij}^{ab} $. Therefore, the above equation can be written as 
$$ \alpha(S_{e^b_j})= \sum_{a \in [K], i \in [N]} S_{e^a_i} \otimes q_{ij}^{ab}. $$

 With respect to the ordering of the edges by $$ \{e_1^1, e_2^1,...e_N^1 ; e_1^2,e_2^2,...e_N^2; ... ; e_1^K, e_2^K, ..., e_N^K\} ,$$  the fundamental matrix formed by the generators of  $ Q_{\tau}^{Lin}(\sqcup_{i=1}^{K} L_{N} )$ is given by $ U:=[(U_{ab})]_{NK \times NK} $, where each $U_{ab}:=(q_{ij}^{ab})_{i,j \in [N]}$ is a $ N \times N $ block matrix.\\
 
\noindent We divide the proof into two parts. In the first part, we derive the relations among the generators of the quantum symmetry group $Q_{\tau}^{Lin}(\sqcup_{i=1}^{K} L_{N} )$. In the second part, we use these relations to identify it with the compact quantum group $U_N^+ \wr_{*} S_K^+$.\\

\noindent $\blacktriangleright$ {\bf Finding the relations among the generators of the quantum symmetry group $ Q_{\tau}^{Lin}(\sqcup_{i=1}^{K} L_{N} )$:}

\noindent To determine the quantum symmetry group $ Q_{\tau}^{Lin}(\sqcup_{i=1}^{K} L_{N} )$, we derive some relations among its generators $ \{q_{ij}^{ab} : i,j \in [N], a,b \in [K]\} $ arising from the action $\alpha$ described above.\\

$ \bullet $ Firstly, since $ F^{\sqcup_{i=1}^{K} L_{N}} = N I_{NK \times NK}$, one can easily observe that $U$ and $U^t$ are both unitary matrices.  Therefore,
\begin{eqnarray}
\sum_{y \in [K],s \in [N]} q_{rs}^{xy} {q_{ts}^{zy}}^* =\delta_{(x,r)(z,t)} \label{u1},\\
\sum_{y \in [K],s \in [N]} {q_{sr}^{yx}}^* {q_{st}^{yz}} =\delta_{(x,r)(z,t)} \label{u2},\\
\sum_{y \in [K],s \in [N]} q_{sr}^{yx} {q_{st}^{yz}}^* =\delta_{(x,r)(z,t)} \label{u3}, \\
\sum_{y \in [K],s \in [N]} {q_{rs}^{xy}}^* {q_{ts}^{zy}} =\delta_{(x,r)(z,t)} \label{u4}.
\end{eqnarray}

$ \bullet $ Next, we will show that for each $a,b \in V'$,
\begin{eqnarray}
\sum_{i \in [N]} {q_{ij}^{ab}}^* {q_{ij}^{ab}}=\sum_{l \in [N]} {q_{kl}^{ab}}{q_{kl}^{ab}}^* \text{ for any }  j,k \in [N],    \label{qs2} \\
\sum_{i \in [N]} {q_{ji}^{ba}}^* {q_{ji}^{ba}}=\sum_{l \in [N]} {q_{lk}^{ba}}{q_{lk}^{ba}}^* \text{ for any }  j,k \in [N].    \label{qs2'}
\end{eqnarray}

We denote the above term in \eqref{qs2} by $ c_{ab} $. Therefore, for all $a,b \in V' $,
\begin{equation}
c_{ab}:= \sum_{i \in [N]} {q_{i1}^{ab}}^* {q_{i1}^{ab}}=\sum_{i \in [N]} {q_{i2}^{ab}}^* {q_{i2}^{ab}}=...=\sum_{i \in [N]} {q_{iN}^{ab}}^* {q_{iN}^{ab}}= \sum_{l \in [N]} {q_{1l}^{ab}}{q_{1l}^{ab}}^*=\sum_{l \in [N]} {q_{2l}^{ab}}{q_{2l}^{ab}}^*=...=\sum_{l \in [N]} {q_{Nl}^{ab}}{q_{Nl}^{ab}}^*.  \label{qs2.1}
\end{equation}    
For fixed $a,b \in V' $,
\begin{equation*}
S_{e}^*S_{e}= \sum_{f: s(f)=b} S_fS_f^* \hspace{0.5cm} {\forall e \in E' \text{ such that } r(e)=b }.
\end{equation*}
Applying the action $\alpha $ to both sides of the above equation, we get
\begin{align*}
\sum_{g \in E'} S_g^*S_g \otimes q_{ge}^*q_{ge}=\sum_{g,h \in E'} S_g S_h^* \otimes \left( \sum_{f: s(f)=b} q_{gf}q_{hf}^* \right)  ~~~~ { \forall e \in E' \text{ such that } r(e)=b } \\
\Rightarrow \sum_{g \in E'} p_{r(g)} \otimes q_{ge}^*q_{ge}=\sum_{g,h \in E'} S_g S_h^* \otimes \left( \sum_{f: s(f)=b} q_{gf}q_{hf}^* \right)  ~~~~ { \forall e \in E' \text{ such that } r(e)=b }.
\end{align*} 
Take any loop $g' \in E'$ with $r{(g')}=s(g')=a$. Now, multiply the above equation by $S_{g'}^* \otimes 1$ from the left and by $S_{g'} \otimes 1 $ from the right, we get
\begin{align*}
&\sum_{g \in E'} S_{g'}^*p_{r(g)}S_{g'} \otimes q_{ge}^*q_{ge}=\sum_{g,h \in E'}  S_{g'}^* S_g S_h^* S_{g'} \otimes \left( \sum_{f: s(f)=b} q_{gf}q_{hf}^* \right)  ~~~~ { \forall e \in E' \text{ such that } r(e)=b } \\
\Rightarrow & \sum_{g: r(g)=s(g)=a} p_a \otimes q_{ge}^*q_{ge} = S_{g'}^* S_{g'} S_{g'}^* S_{g'} \otimes \left( \sum_{f: s(f)=b} q_{g'f}q_{g'f}^* \right)  ~~~~ { \forall e \in E' \text{ such that } r(e)=b }\\
\Rightarrow &~~  p_a \otimes \left( \sum_{g: r(g)=s(g)=a} q_{ge}^*q_{ge} \right)  = p_a \otimes \left( \sum_{f: s(f)=b} q_{g'f}q_{g'f}^* \right)  ~~~~ { \forall e \in E' \text{ such that } r(e)=b }.
\end{align*}
Therefore, for every $ a,b \in V' $,
\begin{equation*}
\left( \sum_{g: r(g)=s(g)=a} q_{ge}^*q_{ge} \right)= \left( \sum_{f: s(f)=b} q_{g'f}q_{g'f}^* \right)  ~~~~ { \forall e,g' \in E' \text{ such that } r(e)=b , r(g')=s(g')=a }.
\end{equation*}
Hence, following our notations, we have 
\begin{equation*}
\sum_{i \in [N]} {q_{e_i^a e_j^b}}^*{q_{e_i^a e_j^b}}= \sum_{l \in [N]} q_{e_k^a e_l^b} { q_{e_k^a e_l^b}}^*
\end{equation*} and Equation \eqref{qs2} follows from the convention $ q_{e_i^a e_j^b}:=q_{ij}^{ab} $ . Applying antipode $\kappa$ to both sides of Equation \eqref{qs2}, we get \eqref{qs2'}. \\

$ \bullet $ Moreover, we claim that for each $ a,b,b' \in V' $ with $ b \neq b' $,
\begin{eqnarray}
q_{ij}^{ab} q_{kl}^{ab'}=0 \hspace{0.5cm} \forall i,j,k,l \in [N] \label{qs3},\\
q_{ji}^{ba} q_{lk}^{b'a}=0  \hspace{0.5cm} \forall i,j,k,l \in [N] \label{qs3'}.
\end{eqnarray}
 Clearly, \eqref{qs3'} follows from \eqref{qs3} simply by taking the antipode $\kappa$ on Equation \eqref{qs3}.
 To show \eqref{qs3}, observe that $S_{e_j^b}S_{e_l^{b'}}=0 $ for all $j,l \in [N]$ and $ b \neq b' $. Applying the action $ \alpha $, we get 
 \begin{equation*}
 \sum_{g,h \in E'} S_g S_h \otimes q_{ge_j^b}q_{he_l^{b'}}=0.
 \end{equation*}
 Since $\{S_g S_h: g, h \in E' \text{ with } r(g)=s(h)\}$ is a linearly independent set (by Lemma 2.3 of \cite{unitaryeasygraph});
 \begin{align*}
  &~~ q_{ge_j^b}q_{he_l^{b'}}=0 ~~~~ \forall g,h \in E' \text{ such that } r(g)=s(h)\in [K] \\
 \Rightarrow  &~~q_{e_i^a e_j^b}q_{e_k^a e_l^{b'}}=0 ~~~~ \forall i,j,k,l, \in [N] \text{ and } a,b,b'\in [K] \text{ but } b \neq b',
 \end{align*}
 i.e., $ q_{ij}^{ab} q_{kl}^{ab'}=0 \hspace{0.3cm} \forall i,j,k,l \in [N] \text{ and } a,b,b'\in [K] \text{ but } b \neq b'$.\\
 
 \noindent Using the relations above, namely Equations $\eqref{u1} - \eqref{qs2'}$, \eqref{qs3} and \eqref{qs3'}, we will derive some more relations among $ \{q_{ij}^{ab}\} $.
 
$ \bullet $ For $a,b,b' \in V'$ with $b \neq b'$ and $i,j,k,l \in [N]$,
 \begin{eqnarray}
 {q_{ij}^{ab}}^* q_{kl}^{ab'}=0   \label{qs4}\\
 {q_{ji}^{ba}}^* q_{lk}^{b'a}=0    \label{qs4'}\\
 q_{ij}^{ab} {q_{kl}^{ab'}}^*=0    \label{qs5}\\
 q_{ji}^{ba} {q_{lk}^{b'a}}^*=0    \label{qs5'}
 \end{eqnarray}

\noindent We will show only Equations \eqref{qs4} and \eqref{qs5}. A similar proof works for Equations \eqref{qs4'} and \eqref{qs5'} respectively. To show Equation \eqref{qs4}, observe that
 \begin{align*}
 {q_{ij}^{ab}}^* q_{kl}^{ab'} = & {q_{ij}^{ab}}^*\left( \sum_{y \in [K],s \in [N]} {q_{is}^{ay}}^* {q_{is}^{ay}}  \right) q_{kl}^{ab'}   ~~~~~~~ [ \text{using } \eqref{u4}]\\
 = &  \sum_{y \in [K],s \in [N]} {\left( q_{is}^{ay} q_{ij}^{ab} \right)}^* \left(q_{is}^{ay} q_{kl}^{ab'} \right) \\
  = &  \sum_{s \in [N]} {\left( q_{is}^{ab} q_{ij}^{ab} \right)}^* \left(q_{is}^{ab} q_{kl}^{ab'} \right) =0  ~~~~~~~ [ \text{using } \eqref{qs3}].
 \end{align*}
Similarly, for Equation \eqref{qs5}, we have
\begin{align*}
q_{ij}^{ab} {q_{kl}^{ab'}}^*=& ~~ q_{ij}^{ab} \left( \sum_{y \in [K],s \in [N]} {q_{is}^{ay}} {q_{is}^{ay}}^* \right)  {q_{kl}^{ab'}}^*     ~~~~~~~ [ \text{using } \eqref{u1}] \\
=& \sum_{y \in [K],s \in [N]} \left( q_{ij}^{ab} {q_{is}^{ay}} \right) {\left( {q_{kl}^{ab'}} {q_{is}^{ay}} \right)}^*  \\
=&  \sum_{s \in [N]} \left( q_{ij}^{ab} {q_{is}^{ab}} \right) {\left( {q_{kl}^{ab'}} {q_{is}^{ab}} \right)}^* =0     ~~~~~~~ [ \text{using } \eqref{qs3}].
\end{align*} 
 
 $\bullet$ For each $a,b \in V'$, $(c_{ab})_{K \times K}$ is a magic unitary matrix such  that 
 \begin{equation}
 q_{ij}^{ab} c_{ab}= c_{ab} q_{ij}^{ab} = q_{ij}^{ab} ~~~~ \forall a,b \in [K] \text{ and } i,j \in [N]      \label{proj}
 \end{equation}

Observe that for $a,a' \in [K]$ with $a \neq a'$,
\begin{align}
c_{ab}c_{a'b}=& \left( \sum_{i \in [N]} {q_{i1}^{ab}}^*{q_{i1}^{ab}}\right)  \left( \sum_{i' \in [N]} {q_{i'1}^{a'b}}^*{q_{i'1}^{a'b}}\right) \nonumber \\
=&   \sum_{i,i' \in [N]} {q_{i1}^{ab}}^*{q_{i1}^{ab}} {q_{i'1}^{a'b}}^*{q_{i'1}^{a'b}} =0 ~~~~[\text{by } \eqref{qs5'}]  \label{ortho}
\end{align}

Since $\sum\limits_{x \in [K]} c_{xb}=1$ (using \eqref{u2}), multiplying both sides by $c_{ab}$, we have
$\sum\limits_{x \in [K]} c_{ab}c_{xb}=c_{ab} \Rightarrow c_{ab}^2=c_{ab} $~~~~[by \eqref{ortho}]. Moreover, clearly $c_{ab}^* =c_{ab}$ . Hence, each $c_{ab}$ is a projection. \\

Since $\sum\limits_{x \in [K]} c_{xb}=1$ (using \eqref{u2}), multiplying $q_{ij}^{ab}$ from the left, we have 
\begin{align*}
&\sum_{x \in [K], i' \in [N]} q_{ij}^{ab} {q_{i'1}^{xb}}^* q_{i'1}^{xb}= q_{ij}^{ab}
\Rightarrow  ~~\sum_{ i' \in [N]} q_{ij}^{ab} {q_{i'1}^{ab}}^* q_{i'1}^{ab}= q_{ij}^{ab} \hspace{0.5cm}[\text{by \eqref{qs5'}}]\\
\Rightarrow & ~~q_{ij}^{ab} \left( \sum_{ i' \in [N]}  {q_{i'1}^{ab}}^* q_{i'1}^{ab} \right) = q_{ij}^{ab} 
\Rightarrow  ~~q_{ij}^{ab} c_{ab} = q_{ij}^{ab}.
\end{align*}
Again, multiplying $q_{ij}^{ab}$ from the right, one can get
\begin{align*}
&\sum_{x \in [K], i' \in [N]}  {q_{i'1}^{xb}}^* q_{i'1}^{xb}q_{ij}^{ab}= q_{ij}^{ab}
\Rightarrow  ~~\sum_{ i' \in [N]}  {q_{i'1}^{ab}}^* q_{i'1}^{ab}q_{ij}^{ab}= q_{ij}^{ab} \hspace{0.5cm}[\text{by \eqref{qs3'}}]\\
\Rightarrow & ~~\left( \sum_{ i' \in [N]}  {q_{i'1}^{ab}}^* q_{i'1}^{ab} \right)q_{ij}^{ab}  = q_{ij}^{ab} 
\Rightarrow  ~~ c_{ab} q_{ij}^{ab}  = q_{ij}^{ab}.
\end{align*}
Hence, \eqref{proj} holds.\\

\noindent Therefore, the generators $\{q_{ij}^{ab} : a,b \in [K] \text{ and } i,j \in [N]\}$ of the quantum symmetry group $Q_{\tau}^{Lin}(\sqcup_{i=1}^{K} L_N)$ satisfy the relations specified in Equations $\eqref{u1} - \eqref{qs2'}$, $\eqref{qs3}- \eqref{qs3'}$, $\eqref{qs4} - \eqref{qs5'}$ and $\eqref{proj}$. Moreover, the coproduct on these generators is given by $\Delta_1(q_{ij}^{ab})=\sum\limits_{c \in [K], k \in [N]} q_{ik}^{ac} \otimes q_{kj}^{cb}$.\\[0.6cm]

\noindent $\blacktriangleright$ {\bf Finding the CQG isomorphism between $Q_{\tau}^{Lin}(\sqcup_{i=1}^{K} L_N)$ and $U_N^+ \wr_{*} S_K^+$ :}

\noindent We will now identify the quantum symmetry group $Q_{\tau}^{Lin}(\sqcup_{i=1}^{K} L_N)$ with the compact quantum group $ (U_N^+ \wr_{*} S_K^+, \Delta_{2} )$, which is described as follows.\\
$U_N^+ \wr_{*} S_K^+$ is defined as the universal $C^*$-algebra generated by the entries of the matrices $ U^{(1)}=(u_{ij}^{1})_{N \times N}, U^{(2)}=(u_{ij}^{2})_{N \times N}, ... , U^{(K)}=(u_{ij}^{K})_{N \times N}, T=(t_{ab})_{K \times K}$ such that
 
\begin{itemize}
\item[(a)] for each $p\in [K]$,  both $U^{(p)}$ and  ${U^{(p)}}^t$ are unitary matrices,
\item[(b)] $T$ is a magic unitary matrix,
\item[(c)] $u_{ij}^{a} t_{ab}= t_{ab} u_{ij}^{a} ~~~~\forall a,b \in  [K], i,j \in [N]$.
\end{itemize}
Moreover, the coproduct $\Delta_2$ is given on generators by $\Delta_2(u_{ij}^a)=\sum\limits_{c'\in [K],k \in [N] } (u_{ik}^a \otimes u_{kj}^{c'})(t_{ac'} \otimes 1)$ and $\Delta_2(t_{ab})=\sum\limits_{c \in [K]} t_{ac} \otimes t_{cb}$.\\

Define $\phi : Q_{\tau}^{Lin}(\sqcup_{i=1}^{K} L_N) \to  U_N^+ \wr_{*} S_K^+ $ on generators by $\phi(q_{ij}^{ab})= u_{ij}^{a}t_{ab} =: w_{ij}^{ab}. $ We will prove that $\phi$ is a CQG isomorphism in three steps.\\

\noindent {\bf STEP-1:} We first show that $\phi$ is a surjective *-homomorphism from $Q_{\tau}^{Lin}(\sqcup_{i=1}^{K} L_N)$ to $ U_N^+ \wr_{*} S_K^+ $. To do this, we construct a linear, $\tau$-preserving action $\alpha$ such that $(U_N^+ \wr_{*} S_K^+, \Delta_{2})$ acts faithfully on $C^*(\sqcup_{i=1}^{K} L_N)$; in other words, $((U_N^+ \wr_{*} S_K^+, \Delta_2), \alpha)$ is an object of the category $\mathfrak{C}_{\tau}^{Lin}$.\\

\noindent For our convenience, we derive several relations among $\{w_{ij}^{ab}\}$ and $\{t_{ab}\}$, which will later be used to show that $\alpha$ is a well-defined $\tau$-preserving action.\\
For all $x,z \in [K]$ and $i,j \in [N]$, we have the following relations:
\begin{align}
\sum_{y \in [K], s \in [N]} w_{is}^{xy} {w_{js}^{zy}}^*=& \sum_{y \in [K], s \in [N]} u_{is}^{x}t_{xy}\left( u_{js}^{z} t_{zy} \right)^* = \sum_{y \in [K], s \in [N]} u_{is}^{x}t_{xy}t_{zy}{u_{js}^{z}}^*  \nonumber \\
=& \sum_{ s \in [N]} u_{is}^{x}\left( \sum_{y \in [K]}t_{xy}t_{zy}\right) {u_{js}^{z}}^*=\sum_{ s \in [N]} u_{is}^{x}( \delta_{xz}) {u_{js}^{z}}^* ~~[\text{by (b)}] \nonumber \\
=& ~\delta_{xz} \sum_{ s \in [N]} u_{is}^{x} {u_{js}^{z}}^*=
\begin{cases}
\sum\limits_{ s \in [N]} u_{is}^{x} {u_{js}^{x}}^*=\delta_{ij} & if ~~ x=z\\
0 & otherwise  \\
\end{cases} ~~[\text{by (a)}] \nonumber \\
=& \begin{cases}
1 & if ~~ x=z, i=j\\
0 & otherwise
\end{cases}= \delta_{(x,i)(z,j)};      \label{D}
\end{align}

\begin{align}
\sum_{y \in [K], s \in [N]} w_{si}^{yx} {w_{sj}^{yz}}^*= & \sum_{y \in [K], s \in [N]}u_{si}^{y}t_{yx}t_{yz}{u_{sj}^{y}}^*=\sum_{y \in [K], s \in [N]}t_{yx}t_{yz}u_{si}^{y}{u_{sj}^{y}}^* ~~[\text{by (c)}]  \nonumber \\
=& \sum_{y \in [K]} t_{yx}t_{yz}\left( \sum_{s \in [N]} u_{si}^{y}{u_{sj}^{y}}^*\right) = \sum_{y \in [K]} t_{yx}t_{yz} (\delta_{ij}) ~~[\text{by (a)}]  \nonumber \\
=& ~~ \delta_{ij}\delta_{xz} ~~[\text{by (b)}] \nonumber \\
 =& ~~\delta_{(x,i)(z,j)}.   \label{E}
\end{align}

\noindent Similarly, one can show that
\begin{equation}
\sum_{y \in [K], s \in [N]} {w_{si}^{yx}}^* {w_{sj}^{yz}}= \delta_{(x,i)(z,j)} ~~ \text{ and } ~~
\sum_{y \in [K], s \in [N]} {w_{is}^{xy}}^* {w_{js}^{zy}}= \delta_{(x,i)(z,j)}.   \label{F}
\end{equation}

\noindent Furthermore, for all $a,b \in [K]$ and $i,j \in [N]$, we have
\begin{align}
\sum_{l \in[N]}{w_{li}^{ab}}^*{w_{lj}^{ab}}  = \sum_{l \in[N]} \left( {u_{li}^{a}t_{ab}}\right) ^*{u_{lj}^{a}t_{ab}}= \sum_{l \in[N]}t_{ab}{u_{li}^{a}}^*{u_{lj}^{a}}t_{ab} = & t_{ab}\left( \sum_{l \in[N]}{u_{li}^{a}}^*{u_{lj}^{a}}\right) t_{ab}   \nonumber\\
= & \delta_{ij} t_{ab} ~~\text{ [by (a) and (b)]},    \label{A}
\end{align}
and 
\begin{align}
\sum_{l \in[N]}{w_{il}^{ab}}{w_{jl}^{ab}}^* =& \sum_{l \in [N]} u_{il}^{a}t_{ab}{u_{jl}^{a}}^* ~~[\text{by (b)}]  \nonumber\\
 =& \sum_{l \in [N]} t_{ab} u_{il}^{a}{u_{jl}^{a}}^* ~~[\text{by (c)}]   \nonumber\\
 =& ~~ t_{ab} \sum_{l \in [N]}  u_{il}^{a}{u_{jl}^{a}}^*= \delta_{ij} t_{ab} ~~[\text{by (a)}].    \label{B}
\end{align}

\vspace{0.5cm}





We now define an action $\alpha: C^{*}(\sqcup_{i=1}^{K}L_N) \to C^{*}(\sqcup_{i=1}^{K}L_N) \otimes (U_{N}^+ \wr_{*} S_{K}^+)$ on generators by 
\begin{align*}
\alpha(S_{e_{j}^{b}}) := & \sum\limits_{a \in [K], i \in [N]} S_{e_{i}^{a}} \otimes w^{ab}_{ij},\\
\alpha(p_b):= & \sum\limits_{a \in [K]} p_{a} \otimes t_{ab}
\end{align*}
for all $b \in [K]$ and $j \in [N]$. To show the existence of $\alpha$, we verify that the elements of $\{\alpha(S_{e_{j}^{b}}), \alpha(p_c) : j \in [N] \text{ and } b,c \in [K]\}$ satisfy the defining relations (i) and (ii) of Definition \ref{Def_Graph C* algebra} for $C^{*}(\sqcup_{i=1}^{K}L_N)$, as follows:
\begin{align*}
\alpha(S_{e_{j}^{b}})^{*} \alpha(S_{e_{j}^{b}})= \sum\limits_{\substack{a,c \in [K],\\ i,k \in [N]}} S_{e_{i}^{a}}^{*} S_{e_{k}^{c}} \otimes {w_{ij}^{ab}}^* w_{kj}^{cb} = & \sum\limits_{a \in [K], i \in [N]} S_{e_{i}^{a}}^{*} S_{e_{i}^{a}} \otimes {w_{ij}^{ab}}^* w_{ij}^{ab} \text{ [by Proposition \ref{graphprop} (i)]}\\
 =& \sum\limits_{a \in [K]} p_{a} \otimes \left( \sum\limits_{i \in [N]} {w_{ij}^{ab}}^* w_{ij}^{ab} \right) = \sum\limits_{a \in [K]} p_{a} \otimes t_{ab} ~ ~ \text{ [by \eqref{A}] }\\
  = & ~~ \alpha(p_b);
\end{align*}
\begin{align*}
 \sum\limits_{j \in [N]} \alpha(S_{e_{j}^{b}}) \alpha(S_{e_{j}^{b}})^{*} = \sum\limits_{j \in [N]} ~~ \sum\limits_{\substack{a,c \in [K],\\ i,k \in [N]}} S_{e_{i}^{a}} S_{e_{k}^{c}}^{*} \otimes w_{ij}^{ab} {w_{kj}^{cb}}^* = & \sum\limits_{j \in [N]} ~~ \sum\limits_{\substack{a\in [K],\\ i,k \in [N]}} S_{e_{i}^{a}} S_{e_{k}^{a}}^{*} \otimes w_{ij}^{ab} {w_{kj}^{ab}}^* \text{ [by Proposition \ref{graphprop} (iv)]}\\
  = & \sum\limits_{\substack{a \in [K],\\ i,k \in [N]}} S_{e_{i}^{a}} S_{e_{k}^{a}}^{*} \otimes \delta_{ik}t_{ab} ~~ \text{ [by \eqref{B}] }\\
  =& \sum\limits_{a \in [K],i \in [N]} S_{e_{i}^{a}} S_{e_{i}^{a}}^{*} \otimes t_{ab}
  =  \sum\limits_{a \in [K]} \left( \sum\limits_{i \in [N]} S_{e_{i}^{a}} S_{e_{i}^{a}}^{*} \right) \otimes t_{ab}\\
  = & \sum\limits_{a \in [K]} p_a \otimes t_{ab} = \alpha(p_b).
\end{align*}
Since $\alpha$ is a linear action, it is straightforward to verify that $\alpha$ satisfies both the action equation and faithfulness condition.\\
 To show Podle\'s condition, we ensure that both $S_{e_{i}^{a}} \otimes 1$ and  $S_{e_{i}^{a}}^* \otimes 1$ belong to $span[{\alpha(C^{*}(\sqcup_{i=1}^{K}L_N))(1 \otimes U_N^+ \wr_{*} S_K^+)}]$ for all $a \in [K]$ and $i \in [N]$. Indeed, observe that
\begin{align*}
\sum\limits_{c \in [K], s \in [N]} \alpha(S_{e_{s}^{c}})(1 \otimes {w_{is}^{ac}}^*) = & \sum\limits_{\substack{c \in [K],\\ s \in [N]}} \left (\sum\limits_{\substack{p \in [K],\\ k \in [N]}} S_{e_{k}^{p}} \otimes w_{ks}^{pc} \right ) \left (1 \otimes {w_{is}^{ac}}^* \right ) = \sum\limits_{\substack{c,p \in [K],\\ s,k \in [N]}} S_{e_{k}^{p}} \otimes  w_{ks}^{pc} {w_{is}^{ac}}^* \\
=& \sum\limits_{p \in [K], k \in [N]} S_{e_{k}^{p}} \otimes  \left ( \sum\limits_{c \in [K], s \in [N]} w_{ks}^{pc} {w_{is}^{ac}}^* \right )  \\
=& \sum\limits_{p \in [K], k \in [N]} S_{e_{k}^{p}} \otimes \delta_{(p,k) (a,i)}  ~~ \text{ [by \eqref{D}] }\\
=& ~~ S_{e_{i}^{a}} \otimes  1,
\end{align*}
and Equation \eqref{F} similarly ensures that 
\begin{equation*}
	\sum\limits_{c \in [K], s \in [N]} \alpha(S_{e_{s}^{c}}^{*})(1 \otimes {w_{is}^{ac}})=  S_{e_{i}^{a}}^{*} \otimes  1.
\end{equation*}
Moreover, since $\alpha$ is a unital action, $1 \otimes q = \alpha(1)(1 \otimes q) \in \alpha(C^{*}(\sqcup_{i=1}^{K}L_N))(1 \otimes U_N^+ \wr_{*} S_K^+)$ for all $q \in U_{N}^{+} \wr_{*} S_{K}^{+}$.\\
The rest of the argument now follows by adopting standard facts from CQG actions [for further details, one may consult 4.2.3 and 4.2.4 in Section 4 of \cite{Web}].\\

\noindent Finally, since the subspace $\mathcal{V}_{2,+}$ for the graph $\sqcup_{i=1}^{K}L_N$ is spanned by $\{ S_{e_{i}^{a}} S_{e_{j}^{a}}^{*} : a \in [K] \text{ and }i,j \in [N]\}$, one can get the $\tau$-preserving condition from the following:
\begin{align*}
(\tau \otimes id) \alpha(S_{e_{i}^{a}} S_{e_{j}^{a}}^{*}) = \sum\limits_{\substack{c,d \in [K],\\ k, l \in [N]}} \tau(S_{e_{k}^{c}} S_{e_{l}^{d}}^*) ~~ w_{ki}^{ca} {w_{lj}^{da}}^{*} = & \sum\limits_{\substack{c,d \in [K],\\ k,l \in [N]}} \delta_{(c,k)(d,l)} ~~ w_{ki}^{ca} {w_{lj}^{da}}^{*}\\
= & \sum\limits_{c \in [K], k \in [N]} w_{ki}^{ca} {w_{kj}^{ca}}^{*} = \delta_{ij}1 ~~ \text{ [by \eqref{E}] }\\
= & ~~ \tau(S_{e_{i}^{a}} S_{e_{j}^{a}}^{*}) 1.
\end{align*}


Therefore, by the universal property of $Q_{\tau}^{Lin}(\sqcup_{i=1}^{K} L_N)$, there exists a surjective $C^*$-homomorphism $\phi : Q_{\tau}^{Lin}(\sqcup_{i=1}^{K} L_N) \to  U_N^+ \wr_{*} S_K^+ $ such that $\phi(q_{ij}^{ab})=w_{ij}^{ab}= u_{ij}^{a}t_{ab} $. \\

\noindent {\bf STEP-2:} We define a *-homomorphism $ \psi : U_N^+ \wr_{*}S_K^+ \to Q_{\tau}^{Lin}(\sqcup_{i=1}^{K} L_N)$ such that $\psi$ is the inverse of $\phi$.\\ 
To this end, we define $\psi$ on the generators of $U_N^+ \wr_{*}S_K^+$ by $$ \psi(u_{ij}^{a})= \sum\limits_{b \in [K]} q_{ij}^{ab} =: v_{ij}^{a}  ~~\text{ and } ~~ \psi(t_{ab})= \sum\limits_{i \in [N]} {q_{ij}^{ab}}^* {q_{ij}^{ab}} = c_{ab}.$$ To establish the existence of $\psi$, we verify that $\{v_{ij}^{a}, c_{ab}: a,b \in [K] \text{ and } i,j \in [N]\}$ satisfies all three defining relations (a), (b) and (c). \\

\noindent We start by observing that
\begin{align*}
\sum_{k \in [N]} v_{ik}^{a}{v_{jk}^{a}}^*=& \sum_{k \in [N]} ~~~~\sum_{ b,b' \in [K]} q_{ik}^{ab} {q_{jk}^{ab'}}^*\\
=& \sum_{k \in [N], b \in [K]} q_{ik}^{ab} {q_{jk}^{ab}}^* ~~[\text{by \eqref{qs5}}]\\
=& ~~ \delta_{ij} ~~[\text{by \eqref{u1}}], 
\end{align*}
and 
\begin{align*}
\sum_{k \in [N]} {v_{ki}^{a}}^*{v_{kj}^{a}}= & \sum_{k \in [N]} ~~~~\sum_{ b,b' \in [K]} {q_{ki}^{ab}}^* {q_{kj}^{ab'}}\\
=& \sum_{k \in [N], b \in [K]} {q_{ki}^{ab}}^* {q_{kj}^{ab}} ~~[\text{by \eqref{qs4}}]\\
=&\sum_{k \in [N], b \in [K]} q_{ik}^{ab} {q_{jk}^{ab}}^* ~~[\text{by } \eqref{qs2.1}]\\ =& \delta_{ij}. 
\end{align*}

\noindent A similar computation also shows that
\begin{align*}
\sum_{k \in [N]} {v_{ki}^{a}}{v_{kj}^{a}}^*=& ~~\delta_{ij}, \\
\text{ and } \sum_{k \in [N]} {v_{ik}^{a}}^{*}{v_{jk}^{a}}=& ~~ \delta_{ij}.
\end{align*}

\noindent Therefore, the above calculations imply that for each $ p \in [K] $, both $V^{(p)}:= (v_{ij}^{p})_{N \times N}$ and ${V^{(p)}}^t$ are unitary matrices.\\

\noindent Since each $c_{ab}$ is a projection with\\
\begin{align*}
\sum_{x \in [K]}c_{xb}=&  1\\
\text{ and }\sum_{y \in [K]} c_{ay}= &  \sum_{y \in [K]} \sum_{i \in  [N]} {q_{ij}^{ay}}^* {q_{ij}^{ay}}= \sum_{y \in [K]} \sum_{j \in  [N]} {q_{ij}^{ay}} {q_{ij}^{ay}}^* ~~[\text{using \eqref{qs2.1}}]\\=& 1, 
\end{align*}
\noindent $ (c_{ab})_{K \times K} $ is a magic unitary matrix.\\

\noindent Lastly, we show that 
$ v_{ij}^{a} c_{ab} = c_{ab} v_{ij}^{a} ~~\forall a,b \in[K] \text{ and }i,j \in [N] $.
Indeed,
\begin{align*}
v_{ij}^{a}c_{ab}= \sum_{b' \in [K], i' \in [N]} q_{ij}^{ab'}  {q_{i'1}^{ab}}^* {q_{i'1}^{ab}}=&~~ \sum_{ i' \in [N]} q_{ij}^{ab}  {q_{i'1}^{ab}}^* {q_{i'1}^{ab}} ~~[\text{by } \eqref{qs5}]\\ =& ~~ q_{ij}^{ab}c_{ab}=q_{ij}^{ab} ~~[\text{by } \eqref{proj}].
\end{align*}   
Similarly, using \eqref{qs3} and \eqref{proj}, one can obtain that $c_{ab}v_{ij}^{a}=q_{ij}^{ab}$.\\
Hence, $v_{ij}^{a}c_{ab}=c_{ab}v_{ij}^{a}  ~~\forall a,b \in[K] \text{ and }i,j \in [N]$.\\

Again, the consequences of the above calculations show that the elements of $\{v_{ij}^{a}, c_{ab}: a,b \in [K] \text{ and } i,j \in [N] \} \subset Q_{\tau}^{Lin}(\sqcup_{i=1}^{K} L_N)$ satisfy all the defining relations (a), (b), (c) of $ U_N^+ \wr_{*} S_K^+ $.\\
Therefore, by the universal property of $U_N^+ \wr_{*} S_K^+ $, there exists a surjective $C^*$-homomorphism $ \psi: U_N^+ \wr_{*} S_K^+  \to Q_{\tau}^{Lin}(\sqcup_{i=1}^{K} L_N) $ such that $\psi(u_{ij}^{a}) = v_{ij}^{a} =\sum\limits_{b \in [K]} q_{ij}^{ab} $ and $\psi(t_{ab})= c_{ab} = \sum\limits_{i \in [N]} {q_{ij}^{ab}}^* {q_{ij}^{ab}} $.\\

\noindent Moreover, it is evident that $\phi \circ \psi= id_{U_N^+ \wr_{*} S_K^+}$ and $\psi \circ \phi=id_{Q_{\tau}^{Lin}(\sqcup_{i=1}^{K} L_N)}$.\\
Hence, $\phi : Q_{\tau}^{Lin}(\sqcup_{i=1}^{K} L_N) \to  U_N^+ \wr_{*} S_K^+ $ is a $C^*$-isomorphism.\\

\noindent {\bf STEP-3:} Finally, to show $\phi$ is a CQG morphism, it is sufficient to verify that $(\phi \otimes \phi)\Delta_1=\Delta_2\phi$ on the generators. Starting with the left-hand side, we have
\begin{align*}
(\phi \otimes \phi)\Delta_1(q_{ij}^{ab})= (\phi \otimes \phi) \left (\sum\limits_{k \in [N],c \in [K]} q_{ik}^{ac} \otimes q_{kj}^{cb} \right )= \sum\limits_{k \in [N],c \in [K]} \phi(q_{ik}^{ac}) \otimes \phi(q_{kj}^{cb})= \sum\limits_{k \in [N],c \in [K]} u_{ik}^{a}t_{ac} \otimes u_{kj}^{c}t_{cb}.
\end{align*}
On the other hand, the right-hand side gives  
\begin{align*}
\Delta_2 \phi(q_{ij}^{ab}) =& ~~ \Delta_2(u_{ij}^{a}t_{ab})=\Delta_2(u_{ij}^{a})\Delta_2(t_{ab})=\left[  \sum\limits_{c'\in [K],k \in [N] } (u_{ik}^a \otimes u_{kj}^{c'})(t_{ac'} \otimes 1)\right] \left[ \sum\limits_{c \in [K]} (t_{ac} \otimes t_{cb}) \right] \\
=& ~~ \sum\limits_{c,c'\in [K],k \in [N] } (u_{ik}^a \otimes u_{kj}^{c'})(t_{ac'} \otimes 1) (t_{ac} \otimes t_{cb}) =  \sum\limits_{c'\in [K],k \in [N] } (u_{ik}^a \otimes u_{kj}^{c'})\sum\limits_{c\in [K] }(t_{ac'}t_{ac} \otimes t_{cb})\\
=& ~~   \sum\limits_{c'\in [K],k \in [N] } (u_{ik}^a \otimes u_{kj}^{c'})(t_{ac'} \otimes t_{c'b})=\sum\limits_{c'\in [K],k \in [N] } (u_{ik}^a t_{ac'} \otimes u_{kj}^{c'} t_{c'b}).
\end{align*}
Thus, $(\phi \otimes \phi)\Delta_1=\Delta_2\phi$, and the proof is complete.
 \end{proof}  

\begin{rem}\label{remfreewreath}
	For a connected graph $ \Gamma $, if we consider $K$ disjoint copies of $ \Gamma $, namely $(\sqcup_{i=1}^{K} ~~ \Gamma )$, then the quantum automorphism group of $K$ disjoint copies  $\Gamma$ in the sense of Banica (respectively, Bichon) is isomorphic to the free wreath product of the quantum automorphism group of $\Gamma$ in the sense of Banica (respectively, Bichon) with $S_K^+$ (see \cite{Banver}, \cite{Bichonwreath} for more details), i.e.
	\begin{equation*}
	QAut_{Ban}(\sqcup_{i=1}^{K} ~ \Gamma) \cong  QAut_{Ban}(\Gamma) \wr_* S_K^+
	\end{equation*} 
	and 
	\begin{equation*}
	QAut_{Bic}(\sqcup_{i=1}^{K} ~ \Gamma) \cong  QAut_{Bic}(\Gamma) \wr_* S_K^+.
	\end{equation*}
	Though Theorem \ref{thm2} ensures that a similar result is also true for the graph $L_n$ (i.e., for the underlying Cuntz algebra $\mathcal{O}_n $) in the context of graph $C^*$-algebra, the same is not always true for an arbitrary graph in the category discussed above. In the next subsection, we will provide a counter-example to show the above claim. 
\end{rem}

Now, we will justify the analogous results in terms of the categories $\mathfrak{C}_{\oplus KMS}^{Lin}$ and $\mathfrak{C}_{\mathfrak{F}}(\sqcup_{i=1}^{K} L_{N})$. 
\begin{thm} 
	Let $\{\mathcal{O}_{N}\}_{i=1}^{K}$ be a finite family of isomorphic Cuntz algebras (with $N$ generators) whose underlying graphs are $L_{N}$. Then $ Q_{\oplus KMS }^{Lin}(\sqcup_{i=1}^{K} L_{N} ) \cong Q_{\tau }^{Lin}(\sqcup_{i=1}^{K} L_{N} ) \cong  U_{N}^{+} \wr_{*} S_{K}^+ \cong Q_{KMS}^{Lin}(L_{N}) \wr_{*} S_{K}^+ .$
\end{thm}
\begin{proof}
    In this case,  since ${KMS}_{ln\rho(A(\sqcup_{i=1}^{K} ~L_{N}))} = \oplus_{i=1}^{K}  {KMS}_{ln \rho(A(L_{N}))} $ (by Corollary \ref{directkmseq}), $\mathfrak{C}_{KMS}^{Lin}$ and $\mathfrak{C}_{\oplus KMS}^{Lin}$ coincide for $\oplus_{i=1}^{K} \mathcal{O}_N$.
	Hence, by Theorem  \ref{kmsthm1}, $\mathfrak{C}_{\oplus KMS}^{Lin}$ has a universal object. Now, using Corollary \ref{kmscor1}, we can say that the universal object of $\mathfrak{C}_{\oplus KMS}^{Lin}$ is isomorphic to the universal object of $\mathfrak{C}_{\tau}^{Lin}$. Hence, the result follows. \\
\end{proof}

\begin{thm} \label{orthothm2}
Let $\{\mathcal{O}_{N}\}_{i=1}^{K}$ be a finite family of isomorphic Cuntz algebras (with $N$ generators) whose underlying graphs are $L_{N}$. Then $Q_{\mathfrak{F}}({C^*(\sqcup_{i=1}^{K}L_{N}))} \cong  Q_{\oplus KMS}^{Lin}(\sqcup_{i=1}^{K} L_{N} ) \cong *_{i=1}^{K} U_{N}^{+} \cong Q_{\mathfrak{F}_{i}}({C^*(L_{N}))} \wr_{*} S_{K}^{+}$.
\end{thm}
\noindent Recall that the fundamental representation $q=(q_{ij}^{ab})_{NK \times NK}$ of CMQG  $ (Q_{\oplus KMS }^{Lin}(\sqcup_{i=1}^{K} L_{N}),q ) \cong (Q_{\tau }^{Lin}(\sqcup_{i=1}^{K} L_{N}),q)$ satisfies Equations $\eqref{u1}-\eqref{u4}$, which imply $q$ is unitary. Now, the observations mentioned in Observation \ref{orthothm1rem} ensure that exactly the same arguments as in Theorem \ref{orthothm1} work to show the first isomorphism mentioned in the above theorem.

\subsection{Counter Example} \label{counterex}
In this subsection, we will provide a graph $\Gamma$ to show that
$ Q_{\tau}^{Lin}(\sqcup_{i=1}^{K} \Gamma )$ is not the same as  $Q_{\tau}^{Lin}(\Gamma) \wr_{*} S_{K}^+ $ in general.\\

\noindent Consider the graph $P_1 \sqcup P_1$, which consists of a disjoint union of two distinct edges $e_1$ and $e_2$ (see Figure \ref{P1P1}). Here, $P_1$ represents a simple directed path of length 1 (which is effectively a directed edge). Note that $C^*(P_1 \sqcup P_1)$ is $C^*$-isomorphic to $ M_2(\mathbb{C}) \oplus M_2(\mathbb{C})$. 

\begin{center}
	\begin{figure}[htpb]
		\begin{tikzpicture}
		\draw[fill=black] (0,0) circle (2pt) node[anchor=north]{${1}$};
		\draw[fill=black] (2,0) circle (2pt) node[anchor=north]{${2}$};
		\draw[fill=black] (4,0) circle (2pt) node[anchor=north]{${3}$};
		\draw[fill=black] (6,0) circle (2pt) node[anchor=north]{${4}$};
		
		\draw[black,thick] (0,0)-- node{\rmidarrow}  node[above]{$e_{1}$} (2,0);
		\draw[black,thick] (4,0)-- node{\rmidarrow}  node[above]{$e_{2}$} (6,0);
		
		\end{tikzpicture}
		\caption{$P_1 \sqcup P_1$} \label{P1P1}
	\end{figure}
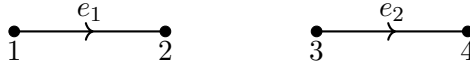
\end{center}   
Again, Proposition 3.4 of \cite{unitaryeasygraph} tells us that $Q_{\tau}^{Lin}(P_1 \sqcup P_1) \approx SH_{2}^{\infty +}$ with respect to their standard fundamental representations, as it is mentioned in Subsections \ref{CQG} and \ref{QSgraph}. On the other hand, $Q_{\tau}^{Lin}(P_1) \wr_{*} S_2^+ \cong C(S^1) \wr_{*} S_2^+ \cong H_{2}^{\infty +}$. But clearly, $SH_{2}^{\infty +}$ is not identical to $H_{2}^{\infty +}$.

\begin{rem}
	Though from Remark \ref{remfreewreath}, we observe that the quantum automorphism group (in the sense of Banica or Bichon) of $K$ disjoint copies of connected graph $\Gamma$  is isomorphic to the free wreath product of the quantum automorphism group of $\Gamma$ with $S_K^+$; Counter Example \ref{counterex} tells us that a similar result may not be true if we move from the context of graph to graph $C^*$-algebra. Now we are ready to ask the following question:\\
	Can one classify the graph $C^*$-algebras $C^*(\Gamma)$ for which 
	$ Q_{\tau}^{Lin}(\sqcup_{i=1}^{K} \Gamma ) \cong Q_{\tau}^{Lin}(\Gamma) \wr_{*} S_{K}^+ $?\\
	(Note that the class is non-empty due to Theorem \ref{thm2}.)
\end{rem}  
 
\begin{rem}
Using similar ideas to what we have used in the main theorems in Sections 3 and 4, one can extend the results to the quantum symmetry of the direct sum of any family of Cuntz algebras.\\
If there are $k_i$ copies of $L_{n_i}$ for all $i \in [m]$, where all $n_i$'s are distinct, then the quantum symmetries of the direct sum Cuntz algebras, namely $Q_{\tau}^{Lin}({\sqcup_{i=1}^{m} ~ \sqcup_{k=1}^{k_i}} L_{n_i})$, $Q_{\oplus KMS}^{Lin}({\sqcup_{i=1}^{m} ~ \sqcup_{k=1}^{k_i}} L_{n_i})$ and $Q_{\mathfrak{F}}(C^*({\sqcup_{i=1}^{m} ~ \sqcup_{k=1}^{k_i}} L_{n_i}))$ (with respect to the categories mentioned before), are isomorphic to $*_{i=1}^{m}  (U_{n_i}^+ \wr_* S_{k_i}^+)$.
\end{rem}

\section*{Acknowledgements}
\noindent The first author acknowledges the financial support from the Department of Science and Technology, India (DST/INSPIRE/03/2021/001688). Both the authors also acknowledge the support from DST-FIST (File No. SR/FST/MS-I/2019/41). We would also like to thank the anonymous referee for his/her useful comments and suggestions on the earlier version of the paper.

\vspace{1cm}

\raggedright{Ujjal Karmakar} \hfill                     
{Arnab Mandal}\\
{Presidency University} \hfill
{Presidency University}\\
{College Street, Kolkata-700073}  \hfill
{College Street, Kolkata-700073}\\
{West Bengal, India}  \hfill
{West Bengal, India}\\
{Email: \email{mathsujjal@gmail.com}} \hfill
{Email: \email{arnab.maths@presiuniv.ac.in}}

\end{document}